\documentclass[11pt,a4paper]{article}
\usepackage[utf8]{inputenc}
\usepackage[T1]{fontenc}
\usepackage{lmodern}
\usepackage[english]{babel}
\usepackage[margin=2.5cm]{geometry}
\usepackage{graphicx}
\usepackage{booktabs}
\usepackage{natbib}
\usepackage{amsmath,amssymb}
\usepackage{setspace}
\usepackage{caption}

\usepackage{amsthm,mathtools,bm}
\usepackage{microtype}
\usepackage{longtable}
\usepackage{float}
\usepackage{needspace}
\usepackage{array}
\usepackage{xcolor}
\usepackage[colorlinks=true,citecolor=blue,linkcolor=blue,urlcolor=blue]{hyperref}

\setcitestyle{round}
\renewcommand{\footnotesize}{\fontsize{9}{11}\selectfont}

\newtheorem{theorem}{Theorem}
\newtheorem{proposition}{Proposition}
\newtheorem{lemma}{Lemma}
\newtheorem{coro}{Corollary}
\newtheorem{rem}{Remark}

\newcommand{\E}{\mathbb{E}}
\newcommand{\PP}{\mathbb{P}}
\newcommand{\Var}{\operatorname{Var}}
\newcommand{\Cov}{\operatorname{Cov}}
\newcommand{\R}{\mathbb{R}}
\newcommand{\D}{\xrightarrow{\mathcal D}}
\newcommand{\Pto}{\xrightarrow{\mathbb P}}
\newcommand{\Ltwo}{\xrightarrow{L^2}}
\newcommand{\Op}{\mathrm O_{\mathbb P}}
\newcommand{\op}{\mathrm o_{\mathbb P}}
\newcommand{\diag}{\operatorname{diag}}

\title{Consistent intercept estimation and inference for unit-root INAR(2) processes}
\author{
    Yang Lu\thanks{Corresponding author: yang.lu@concordia.ca} \\
    Department of Mathematics and Statistics \\
    Concordia University \\
    Montreal, Canada
    \and
    M\'arton Isp\'any \\
    University of Debrecen \\
    Hungary
}
\date{}

\begin{document}
\maketitle

\begin{abstract}
\citet{barczy2014asymptotic} showed that the ordinary least-squares (OLS)
estimator of the innovation mean is inconsistent for a unit-root INAR(2)
process. We construct a consistent
intercept estimator using inverse-time weighted least squares (WLS) and derive
its mixed-rate asymptotics: the intercept estimator is asymptotically normal at
rate $\sqrt{\log n}$, with a Gaussian limit independent of the autoregressive
limits, while the autoregressive estimators retain their OLS rates. We develop
an OLS unit-root test calibrated with WLS nuisance estimates. Under the
maintained unit root, residual-score Gaussian intervals for the innovation
mean and long-run drift remain asymptotically valid after test
nonrejection. Simulations show that WLS reduces root mean squared error for
both quantities relative to OLS, with further gains from imposing the unit
root. An application to two Canadian flood inventories shows that conclusions
about persistence depend on the inventory and weighting offset.
\end{abstract}



\textbf{Acknowledgements:} We thank Mátyás Barczy for helpful comments.

\section{Introduction}
Count time series often arise in settings where the underlying environment
evolves, making a stationary specification with a constant long-run mean
implausible. Natural-disaster frequencies are a salient example: shifts in
climate, exposure, and reporting practice can produce persistent changes in
the recorded frequencies of floods, wildfires, and storms
\citep{pei2023forecasting}.
\needspace{4\baselineskip}
The literature on nonstationary count processes remains relatively limited,
with recent contributions focusing mainly on first-order Markov models
\citep{liu2023asymptotic,barreto-souza2023,peng2024note,lu2026intercept,lu2026stateweights,unitrootaffine}.
Beyond this class, \citet{michel2020limiting} studies a unit-root
INGARCH(1,1) process, with a focus on probabilistic properties rather than
parameter estimation and inference. \citet{barczy2014asymptotic} (henceforth
BIP) study parameter estimation for the unit-root INAR(2) model, showing in
particular that the ordinary least-squares (OLS) intercept estimator is
inconsistent.
Consider the INAR(2) model
\[
X_k=\alpha\circ X_{k-1}+\beta\circ X_{k-2}+\varepsilon_k,
\qquad
\alpha,\beta\in(0,1),
\qquad
\varrho:=\alpha+\beta\leq1,
\]
where $\circ$ denotes binomial thinning. All Bernoulli variables used in the
binomial thinnings are mutually independent and independent of the i.i.d.
sequence $(\varepsilon_k)$ of nonnegative integer-valued innovations with
positive mean $\mu:=\E(\varepsilon_1)$.
This inconsistency matters because, in the unit-root case $\varrho=1$,
$\mu$ enters the unit-root scaling limit and, together with $\beta$,
determines the long-run drift, defined as the limiting expected increment:
\begin{equation}
g:=\lim_{k\to\infty}\{\E(X_k)-\E(X_{k-1})\}
=\frac{\mu}{1+\beta}.
\label{eq:long-run-drift}
\end{equation}

Related intercept-estimation problems arise in first-order integrated
Galton--Watson processes. \citet{wei1990estimation} established consistency
of state-weighted WLS, but derived the intercept's asymptotic distribution
only under transience. \citet{lu2026intercept} introduced deterministic
inverse-time weighting, obtaining a common Gaussian intercept limit across
the transient and recurrent regimes. Extending this approach to unit-root
INAR(2) requires accommodating both thinning randomness and a stable
autoregressive direction. This combination is absent from unit-root INAR(1),
where thinning is deterministic.

In this paper, we develop a unified procedure for estimation, unit-root
testing, and inference in the INAR(2) model using inverse-time weighting.
We propose an unrestricted WLS estimator based on the inverse-time
weights $(k+\ell)^{-1}$ for a fixed weighting offset $\ell\geq0$,
and establish a joint mixed-rate limit: the intercept estimator is consistent
and asymptotically normal at rate $\sqrt{\log n}$, while the persistence and
individual autoregressive estimators retain their OLS rates. The intercept
limit is independent of the autoregressive limits.

We also consider constrained WLS estimation of $(\beta,\mu)$ under the
unit-root restriction. When the restriction is correct, the constrained and
unrestricted estimators have the same joint first-order limit, and we identify
the leading difference between the two intercept estimators. Under a fixed
stationary alternative, unrestricted WLS remains consistent, whereas the
constrained estimator converges to a pseudo-true nuisance pair whose intercept
is zero. For unit-root testing, oracle comparisons favor the OLS statistic,
whose nonstandard null distribution we calibrate using WLS nuisance estimates.
The resulting feasible test is asymptotically valid under the unit root and
consistent against fixed stationary alternatives.

We then construct residual-score Gaussian inference for the innovation mean
and long-run drift under the maintained unit-root model, and show that
test nonrejection leaves its first-order validity intact. Simulations confirm
that WLS improves intercept estimation, with further gains when the unit-root
restriction is correct, and that residual-score intervals cover better than
plug-in intervals.

In an application to two Canadian flood inventories, the test rejects the
unit-root null for the Canadian Disaster Database at the reference offset
but not at the alternative offset. It does not reject for Historical Flood
Events at either offset. Conclusions about persistence therefore depend on
both the inventory and the weighting offset.

The remainder of the paper is organized as follows. Section~2 develops the
unrestricted and constrained WLS estimators, examines weight concentration and
offset sensitivity, and derives the estimators' probability limits under fixed
stationary alternatives. Section~3 treats unit-root testing and calibration,
and Section~4 inference under the maintained unit-root model. Section~5 reports
Monte Carlo evidence for estimation, testing, and inference. Section~6 applies
the methods to the flood data, and Section~7 concludes. All
proofs and supporting technical results are gathered in the Appendix.

\section{Weighted least-squares estimation}
\label{sec:wls-estimation}

Throughout the rest of the paper we take the initial condition
$(X_{-1},X_0)\in\mathbb Z_+^2$ to be fixed, impose the moment condition
$\E(\varepsilon_1^8)<\infty$ used by BIP, and treat the weighting offset
$\ell\geq0$ as fixed and chosen before the data are examined.

Following BIP, we use the parameterization $(\varrho,\beta,\mu)$ and define
\[
V_k:=X_k-X_{k-1},
\qquad
M_k:=X_k-\alpha X_{k-1}-\beta X_{k-2}-\mu.
\]
The sequence $(M_k)$ is a martingale difference sequence with respect to
the natural filtration $\mathcal F_k:=\sigma(X_{-1},X_0,\ldots,X_k)$, and the
model can be written as
\begin{equation}
X_k=\varrho X_{k-1}-\beta V_{k-1}+\mu+M_k.
\label{eq:canonical}
\end{equation}
In this parameterization, $\varrho$ measures total persistence and $-\beta$ is
the coefficient of the lagged difference. The OLS and unrestricted WLS
estimators below use equation~\eqref{eq:canonical}, with regressor vector
$(X_{k-1},-V_{k-1},1)$. The OLS estimator is defined by
\[
(\widehat\varrho_n,\widehat\beta_n,\widehat\mu_n)
=\arg \min_{(\varrho,\beta,\mu)\in\R^3}
\sum_{k=1}^n
\bigl(X_k-\varrho X_{k-1}+\beta V_{k-1}-\mu\bigr)^2.
\]
At $\varrho=1$, BIP show that $\widehat\varrho_n$ and
$\widehat\beta_n$ converge at rates $n$ and $\sqrt n$, respectively, whereas
the OLS estimator
$\widehat\mu_n$ of $\mu$ is inconsistent.
\subsection{The unrestricted WLS estimator}
\label{sec:unrestricted-wls}

For a fixed weighting offset $\ell\geq0$, write
$w_{k,\ell}:=(k+\ell)^{-1}$. The unrestricted WLS estimator of
$(\varrho,\beta,\mu)$ is
\begin{equation}
(\widetilde\varrho_{n,\ell},\widetilde\beta_{n,\ell},
 \widetilde\mu_{n,\ell})
=\arg \min_{(\varrho,\beta,\mu)\in\R^3}
\sum_{k=1}^n w_{k,\ell}
\bigl(X_k-\varrho X_{k-1}+\beta V_{k-1}-\mu\bigr)^2.
\label{eq:wls-criterion}
\end{equation}
The inverse-time form is suggested by the conditional variance: under the
unit-root restriction,
\[
\Var(X_k\mid\mathcal F_{k-1})
=\alpha\beta(X_{k-1}+X_{k-2})+\Var(\varepsilon_1),
\]
whose expectation grows linearly with $k$.

The normal equations give
\begin{equation}
\begin{pmatrix}
\widetilde\varrho_{n,\ell}-\varrho\\[1mm]
\widetilde\beta_{n,\ell}-\beta\\[1mm]
\widetilde\mu_{n,\ell}-\mu
\end{pmatrix}
=
\bigl(A_{n,\ell}^{(w)}\bigr)^{-1}d_{n,\ell}^{(w)},
\label{eq:normal-canonical}
\end{equation}
where
\[
A_{n,\ell}^{(w)}:=\sum_{k=1}^n w_{k,\ell}
\begin{pmatrix}
X_{k-1}^2 & -X_{k-1}V_{k-1} & X_{k-1}\\
-X_{k-1}V_{k-1} & V_{k-1}^2 & -V_{k-1}\\
X_{k-1} & -V_{k-1} & 1
\end{pmatrix},
\qquad
d_{n,\ell}^{(w)}:=\sum_{k=1}^n w_{k,\ell}M_k
\begin{pmatrix}
X_{k-1}\\
-V_{k-1}\\
1
\end{pmatrix}.
\]

\begin{rem}
The WLS estimator is uniquely defined only on the event that the positive
semidefinite matrix $A_{n,\ell}^{(w)}$ is invertible. Under the conditions of
Theorem~\ref{thm:main}, the probability of noninvertibility converges to zero
for every fixed $\ell$.
\end{rem}

The exact-null theory below imposes $\varrho=1$ unless explicitly stated
otherwise.
By \citet[Theorem~3.1]{barczy2011asymptotic}, the rescaled process
$(n^{-1}X_{\lfloor nt\rfloor})_{t\geq0}$ converges weakly to
$(\mathcal X_t)_{t\geq0}$, which solves the stochastic differential equation
\begin{equation}
d\mathcal X_t
=
\frac{\mu}{1+\beta}\,dt
+
\frac{\sqrt{2\alpha\beta\,\mathcal X_t}}{1+\beta}\,dW_t,
\qquad
\mathcal X_0=0,
\label{eq:CIR}
\end{equation}
where $W$ is a standard Wiener process. The limiting diffusion admits the
representation $\mathcal X_t=cR_t$, where $c:=\alpha\beta/[2(1+\beta)^2]$
and $R$ is a squared-Bessel process of dimension
$\delta:=2\mu(1+\beta)/(\alpha\beta)>0$, satisfying
\begin{equation}
dR_t=\delta\,dt+2\sqrt{R_t}\,dW_t,
\qquad R_0=0.
\label{eq:squared-bessel-process}
\end{equation}
The limiting process is transient for $\delta>2$ and recurrent for
$0<\delta\leq2$. We distinguish the boundary case $\delta=2$, which does not
hit zero from a positive starting point. Throughout the paper, we refer to
$\delta>2$, $\delta=2$, and $0<\delta<2$ as the transient, boundary, and
recurrent regimes, respectively.

Write $H_{n,\ell}:=\sum_{k=1}^n(k+\ell)^{-1}$ for the sum of WLS weights.
For every fixed $\ell\geq0$, $H_{n,\ell}=\log n+\mathrm O(1)$, so
$\sqrt{H_{n,\ell}}$ is asymptotically equivalent to $\sqrt{\log n}$.

\begin{theorem}
\label{thm:main}
Under the standing assumptions with $\varrho=1$,
\begin{equation}
\begin{pmatrix}
n(\widetilde\varrho_{n,\ell}-1)\\[2mm]
\sqrt n\,(\widetilde\beta_{n,\ell}-\beta)\\[2mm]
\sqrt{H_{n,\ell}}\,(\widetilde\mu_{n,\ell}-\mu)
\end{pmatrix}
\D
\begin{pmatrix}
\Psi_w\\[2mm]
\Xi_w\\[2mm]
Z
\end{pmatrix},
\label{eq:main-canonical}
\end{equation}
where
\begin{equation}
\begin{aligned}
\Psi_w
&:=\frac{\sqrt{2\alpha\beta}\int_0^1 t^{-1}\mathcal X_t^{3/2}\,dW_t}{I_2},
& I_2&:=\int_0^1 \frac{\mathcal X_t^2}{t}\,dt,
\\
\Xi_w
&:=\sqrt{\alpha(1+\beta)}\,
\frac{\int_0^1 t^{-1}\mathcal X_t\,d\widetilde W_t}{I_1},
& I_1&:=\int_0^1 \frac{\mathcal X_t}{t}\,dt.
\end{aligned}
\label{eq:wls-limit-functionals}
\end{equation}
Here $\widetilde W$ is a standard Wiener process independent of $(\mathcal X,W)$.
The random variable $Z$ is independent of $(\mathcal X,W,\widetilde W)$ and
follows a $\mathcal N(0,\sigma_\mu^2)$ distribution, where
\[
\sigma_\mu^2:=\frac{2\alpha\beta\mu}{1+\beta}
=\frac{2(1-\beta)\beta\mu}{1+\beta}.
\]
\end{theorem}
	 
 \begin{rem}
 Squared-Bessel self-similarity gives
 $\E(\mathcal X_t^j)=\mathrm O(t^j)$ for $j=1,2,3$ as $t\downarrow0$.
 These bounds ensure that all integrals in
 \eqref{eq:wls-limit-functionals} are well defined at the origin.
 \end{rem}

\needspace{3\baselineskip}
Inverse-time weighting therefore restores consistency of the intercept
estimator, at the price of the slow $\sqrt{\log n}$ rate.

Finally, since $\alpha=\varrho-\beta$ and $\varrho$ and $\beta$ are estimated
at different rates,
\[
\sqrt n\{(\widetilde\varrho_{n,\ell}-\widetilde\beta_{n,\ell})-
\alpha\}\D-\Xi_w.
\]

\subsection{The constrained WLS estimator}
\label{sec:constrained-wls}

Imposing the unit-root restriction $\varrho=1$ in
\eqref{eq:wls-criterion} gives the constrained WLS estimator
\begin{equation}
(\widetilde\beta_{n,\ell}^c,\widetilde\mu_{n,\ell}^c)
:=\arg \min_{(\beta,\mu)\in\R^2} \sum_{k=1}^n w_{k,\ell}
\bigl(V_k+\beta V_{k-1}-\mu\bigr)^2,
\label{eq:wls-criterion-constrained}
\end{equation}
which regresses the differenced series on its own lag. The following theorem
shows that imposing the restriction does not
change the joint first-order asymptotic distribution of the estimators of
$(\beta,\mu)$.

\begin{theorem}[First-order limit of constrained WLS]
\label{thm:constrained-wls}
Under the assumptions of Theorem~\ref{thm:main}, the constrained WLS
estimator satisfies
\[
\begin{pmatrix}
\sqrt n\,(\widetilde\beta_{n,\ell}^c-\beta)\\[1mm]
\sqrt{H_{n,\ell}}\,(\widetilde\mu_{n,\ell}^c-\mu)
\end{pmatrix}
\D
\begin{pmatrix}
\Xi_w\\[1mm]
Z
\end{pmatrix}.
\]
\end{theorem}

\subsection{Comparison of unrestricted and constrained WLS}
\label{sec:constraint-estimation}

The next theorem characterizes the leading difference between the constrained
and unrestricted intercept estimators, providing a stochastic comparison of
first-order equivalent estimators in the sense of \citet{robinson1988stochastic}.

\begin{theorem}[Second-order limit of the intercept-estimator difference]
\label{thm:intercept-second-order}
Under the assumptions of Theorem~\ref{thm:main}, the joint convergence
in~\eqref{eq:main-canonical} extends to
\[
\begin{pmatrix}
n(\widetilde\varrho_{n,\ell}-1)\\[1mm]
\sqrt n\,(\widetilde\beta_{n,\ell}-\beta)\\[1mm]
\sqrt{H_{n,\ell}}\,(\widetilde\mu_{n,\ell}-\mu)\\[1mm]
H_{n,\ell}(\widetilde\mu_{n,\ell}^c-\widetilde\mu_{n,\ell})
\end{pmatrix}
\D
\begin{pmatrix}
\Psi_w\\[1mm]
\Xi_w\\[1mm]
Z\\[1mm]
I_1\Psi_w
\end{pmatrix}.
\]
\end{theorem}

The difference between the two intercept estimators is thus second order
relative to their individual estimation errors, which are both
$\Op(H_{n,\ell}^{-1/2})$. Since
$H_{n,\ell}\sim\log n$, the gap closes only logarithmically, so it remains
visible at empirical sample sizes.

\subsection{Weight concentration and fixed-offset sensitivity}
\label{sec:fixed-offsets}

Theorem~\ref{thm:main} permits any fixed initial condition. If the available
historical record is $Y_1,\ldots,Y_N$, inference can therefore be conditional
on its first two observations through the relabeling
\[
X_{-1}:=Y_1,\qquad X_0:=Y_2,\qquad
X_k:=Y_{k+2},\quad k=1,\ldots,n,\qquad n:=N-2.
\]
Thus the raw transition index $j=3,\ldots,N$ corresponds to $k=j-2$. A
raw-time weight $1/(j+\ell_{\mathrm{raw}})$ is therefore
$1/\{k+(\ell_{\mathrm{raw}}+2)\}$, so
$\ell_{\mathrm{reindexed}}=\ell_{\mathrm{raw}}+2$.
In particular, weighting by raw observation time $1/j$ corresponds to
$\ell=2$ after reindexing, whereas resetting the weight clock at the first
fitted transition gives $\ell=0$ and raises the first transition's weight from
$1/3$ to $1$. We use $\ell=0$ as the reference specification.

The offset also controls how much of the total weight falls on the earliest,
smallest counts. Under weights $w_{k,\ell}$, the first $m$ fitted equations
receive the fraction $H_{m,\ell}/H_{n,\ell}$ of the total weight, against $m/n$
under equal weighting. At $n=100$, the first five fitted equations carry $44\%$
of the weight at $\ell=0$ but only $17\%$ at $\ell=10$.
Proposition~\ref{prop:fixed-offset-intercept} compares intercept estimators
computed at different fixed offsets from the same reindexed trajectory.

\needspace{7\baselineskip}
\begin{proposition}[Cross-offset intercept comparison]
\label{prop:fixed-offset-intercept}
Under the conditions of Theorem~\ref{thm:main}, let $\ell_1,\ell_2\geq0$ be
fixed. For estimators computed from the same reindexed trajectory,
\[
\widetilde\mu_{n,\ell_1}-\widetilde\mu_{n,\ell_2}
=\Op((\log n)^{-1}),
\qquad
\widetilde\mu_{n,\ell_1}^c-
\widetilde\mu_{n,\ell_2}^c=\Op((\log n)^{-1}).
\]
\end{proposition}

Changing the fixed offset thus produces a difference that is negligible
relative to the $(\log n)^{-1/2}$ scale of the estimation errors themselves.
Negligibility arrives slowly, however: $(\log n)^{-1}$ is still $0.22$ at
$n=100$ and $0.14$ at $n=1000$.

\subsection{Properties under stationarity}
\label{sec:fixed-stationary}

The next proposition gives the probability limits and convergence rates of
the WLS estimators under a fixed stationary alternative.
\begin{proposition}[WLS under a fixed stationary alternative]
\label{prop:fixed-stationary}
Suppose the assumptions of Theorem~\ref{thm:main} hold, except that
$\varrho<1$. Then
\begin{equation}
 \begin{pmatrix}
  \widetilde\varrho_{n,\ell}\\[-1mm]
  \widetilde\beta_{n,\ell}\\[-1mm]
  \widetilde\mu_{n,\ell}
 \end{pmatrix}
 \xrightarrow{\mathrm{a.s.}}
 \begin{pmatrix}\varrho\\[-1mm]\beta\\[-1mm]\mu\end{pmatrix},
 \qquad
 \begin{pmatrix}
  \widetilde\beta_{n,\ell}^c\\[-1mm]
  \widetilde\mu_{n,\ell}^c
 \end{pmatrix}
 \xrightarrow{\mathrm{a.s.}}
 \begin{pmatrix}
  \displaystyle\beta+\frac{1-\varrho}{2}\\[-1mm]
  0
 \end{pmatrix}.
 \label{eq:fixed-stationary-limits}
\end{equation}
Moreover, for $x:=(X_{-1},X_0)$, the unrestricted error vector multiplied by
$H_{n,\ell}$ converges almost surely to a random vector $\Lambda_{\ell,x}$
with positive-definite covariance. The constrained intercept satisfies
$\widetilde\mu_{n,\ell}^c=\Op(H_{n,\ell}^{-1})$.
\end{proposition}

Thus, under stationarity, unrestricted inverse-time WLS is consistent and
converges at the logarithmic rate $H_{n,\ell}\sim\log n$, slower than the
$\sqrt n$ rate of unweighted conditional least squares \citep{latour1998}.
The constrained WLS estimator is inconsistent.

\section{Unit-root testing}
\label{sec:unit-root-testing}

The unit-root results in Section~\ref{sec:wls-estimation} and BIP take
$\varrho=1$ as a model assumption. We now develop a test of this restriction
against stationary alternatives.

\subsection{Choice of test statistic}
\label{sec:choice-test-statistic}

We have two candidate statistics for testing $H_0:\varrho=1$ against
$H_1:\varrho<1$:
\[
S_n^{\mathrm{OLS}}
:=n(\widehat\varrho_n^{\mathrm{OLS}}-1),
\qquad
T_{n,\ell}^{\mathrm{WLS}}
:=n(\widetilde\varrho_{n,\ell}-1).
\]
Here $\widehat\varrho_n^{\mathrm{OLS}}=\widehat\varrho_n$ denotes the OLS
estimator obtained with equal weights. For
$X_{-1}=X_0=0$, Theorem~2.1 of BIP gives
\begin{equation}
S_n^{\mathrm{OLS}}\D\Psi_\varrho
:=
\frac{\int_0^1\mathcal X_t\,d\mathcal M_t-J_1\mathcal M_1}
{J_2-J_1^2},
\qquad
J_1:=\int_0^1\mathcal X_t\,dt,
\quad
J_2:=\int_0^1\mathcal X_t^2\,dt,
\label{eq:ols-unit-root-limit}
\end{equation}
where
\[
\mathcal M_t:=(1+\beta)\mathcal X_t-\mu t
=\sqrt{2\alpha\beta}\int_0^t\mathcal X_s^{1/2}\,dW_s.
\]
By Theorem~\ref{thm:main}, the WLS statistic likewise satisfies
$T_{n,\ell}^{\mathrm{WLS}}\D\Psi_w$.

To compare the two limits, let $w_k=1$ for OLS and
$w_k=(k+\ell)^{-1}$ for WLS, and denote the corresponding persistence
estimator by $\widehat\varrho_n^{(w)}$. Define
\[
\begin{aligned}
H&:=\sum_{k=1}^n w_k,
& A_n&:=\frac1{n^2}\sum_{k=1}^n w_kX_{k-1}^2,
& B_n&:=\frac1n\sum_{k=1}^n w_kX_{k-1},\\
Q_n&:=\frac1n\sum_{k=1}^n w_kX_{k-1}M_k,
& Z_n&:=\frac1{\sqrt H}\sum_{k=1}^n w_kM_k.
\end{aligned}
\]
Under the unit root, it can be shown from the normal equations that, for both
estimators,
\begin{equation}
n(\widehat\varrho_n^{(w)}-1)
=\frac{Q_n-B_nZ_n/\sqrt H}{A_n-B_n^2/H}
+\Op(n^{-1/2}).
\label{eq:common-persistence-expansion}
\end{equation}

For OLS, $H=n$, and both denominator terms are of order $n$. After division
by $n$, they converge jointly to $J_2$ and $J_1^2$, respectively, so both
contribute to the limiting denominator $J_2-J_1^2$. For WLS, $A_n$ is of
order one, whereas $B_n^2/H$ is of order $1/\log n$. The second term
disappears asymptotically, leaving only $I_2$ in the limiting denominator.

However, $\log n$ grows slowly, so the term omitted from the WLS limiting
denominator can remain appreciable at practical sample sizes. The same issue
arises in the numerator: its intercept adjustment vanishes only at rate
$1/\sqrt{\log n}$, whereas the corresponding OLS adjustment contributes at
first order and is retained in the limit. Thus, although inverse-time
weighting restores intercept consistency, its first-order persistence limit
omits intercept adjustments that disappear only logarithmically.
Inverse-time weighting also gives greater relative influence to early,
small-count observations, where the diffusion approximation may be less
accurate.

\Needspace{7\baselineskip}
We assess the finite-sample consequences by comparing the null distributions
of the OLS and WLS statistics at $n=1000$ with their respective continuous
limits, using the three parameter
designs in Section~\ref{sec:finite-sample} and the reference offset $\ell=0$
for WLS. The designs cover the transient, boundary, and recurrent regimes.
Oracle critical values use the true
nuisance parameters, while all regression coefficients remain estimated.

\begin{figure}[!htbp]
\centering
\includegraphics[width=\textwidth]
{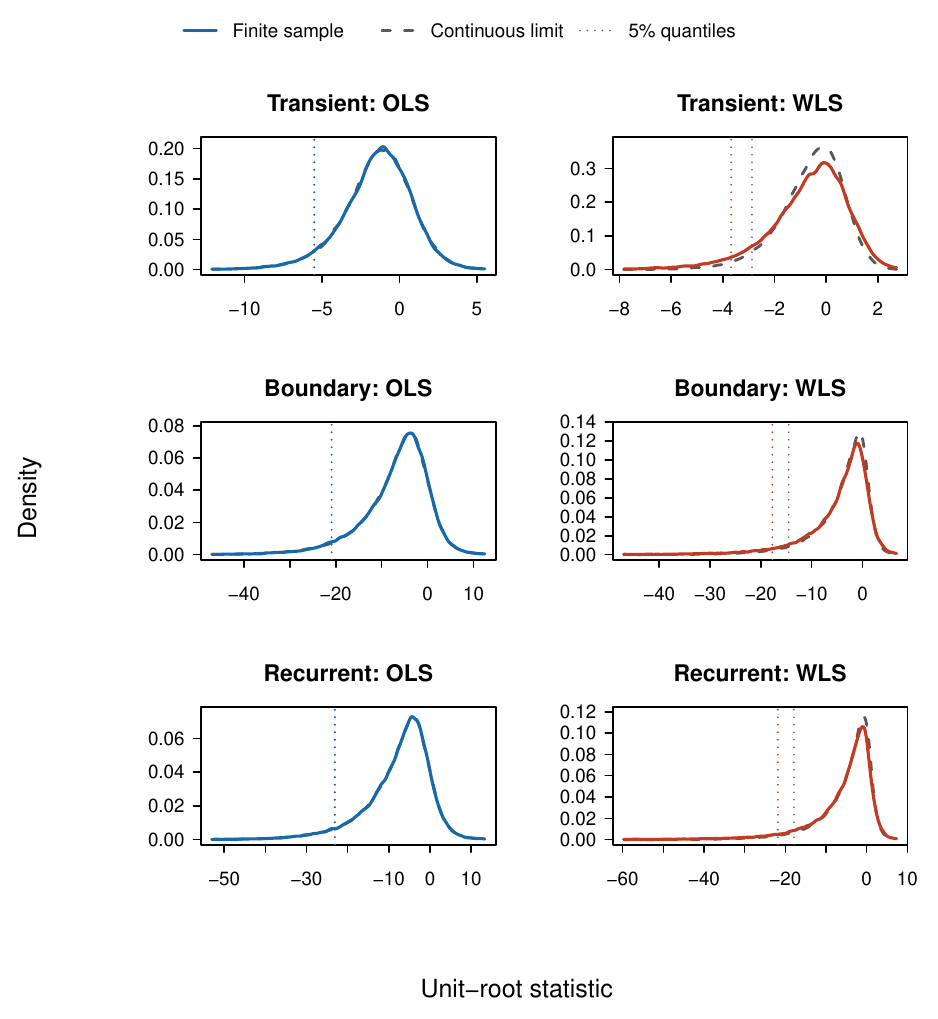}
\caption{Oracle null-law comparison at $n=1000$. Solid colored curves are
kernel density estimates of the finite-sample statistics and dashed gray
curves are those of their continuous limits; dotted lines mark the respective
$5\%$ quantiles. The rows correspond to the transient, boundary, and recurrent
designs.}
\label{fig:oracle-null-density}

\begin{minipage}{0.96\linewidth}
\footnotesize
\emph{Notes:} Finite-sample distributions use $50{,}000$ independent
Poisson-INAR(2) trajectories per design, with OLS and WLS computed on the
same trajectories. Continuous limits
use $100{,}000$ squared-Bessel paths generated by exact transitions on a
$3{,}000$-step grid.
\end{minipage}
\end{figure}

Figure~\ref{fig:oracle-null-density} shows close agreement between the OLS
finite-sample and limiting densities, including their lower $5\%$ quantiles,
across all three designs. The WLS lower quantiles remain to the left of their
limiting counterparts, consistent with the slowly vanishing intercept
adjustments discussed above. These results favor $S_n^{\mathrm{OLS}}$ for
testing, calibrated with WLS nuisance estimates.

\subsection{Nuisance calibration for the OLS test}
\label{sec:ols-unit-root-calibration}
\label{sec:unit-root-calibration}

Constrained WLS is consistent under the null, so a first approach is to
calibrate the OLS statistic by plugging its estimates into the null quantile:
reject the unit-root null when
\[
S_n^{\mathrm{OLS}}
<q_\tau^{\mathrm{OLS}}(\widetilde\beta_{n,\ell}^c,
                      \widetilde\mu_{n,\ell}^c),
\]
where, for $\tau\in(0,1)$, $q_\tau^{\mathrm{OLS}}(\beta,\mu)$ denotes
the lower $\tau$-quantile of $\Psi_\varrho$ under
$(\alpha,\beta)=(1-\beta,\beta)$.

This rule, however, need not be well defined in finite samples: the
constrained-WLS fit may be nonunique, or its estimated nuisance parameters
may lie outside the domain of the null critical-value map. Moreover, under
fixed stationary alternatives, the constrained intercept converges to zero,
so constrained calibration approaches the boundary of that domain.

Let $\Theta_0:=(0,1)\times(0,\infty)$ denote the theoretical nuisance space.
A constrained- or unrestricted-WLS fit is admissible when it is unique and
its fitted nuisance pair belongs to $\Theta_0$. The limiting distribution of
the OLS statistic can be parameterized by $\beta$ and the squared-Bessel
dimension $\delta$, which for a nuisance pair $(b,m)\in\Theta_0$ is
\[
\delta(b,m):=\frac{2m(1+b)}{(1-b)b}.
\]
For an admissible constrained-WLS fit, define the estimated dimension by
\[
\widetilde\delta_{n,\ell}^C
:=\delta(\widetilde\beta_{n,\ell}^c,
         \widetilde\mu_{n,\ell}^c),
\]
and set $\eta_{n,\ell}:=H_{n,\ell}^{-1/2}$ as the selection threshold.

On their respective admissibility events, the two fits yield the raw
plug-in critical values
\[
q_{\tau,n,\ell}^{\mathrm{OLS},C}
:=q_\tau^{\mathrm{OLS}}(\widetilde\beta_{n,\ell}^c,
\widetilde\mu_{n,\ell}^c),
\qquad
q_{\tau,n,\ell}^{\mathrm{OLS},W}
:=q_\tau^{\mathrm{OLS}}(\widetilde\beta_{n,\ell},
\widetilde\mu_{n,\ell}).
\]
The superscripts $C$ and $W$ identify constrained- and unrestricted-WLS
nuisance calibration, respectively. We use constrained calibration when its
estimated dimension exceeds the threshold and unrestricted calibration
otherwise, provided the corresponding fit is admissible. Formally, the
feasible OLS critical value is
\begin{equation}
q_{\tau,n,\ell}^{\mathrm{OLS}}
:=
\begin{cases}
q_{\tau,n,\ell}^{\mathrm{OLS},C},
  &\text{the constrained fit is admissible and }
   \widetilde\delta_{n,\ell}^C>\eta_{n,\ell},\\[1mm]
q_{\tau,n,\ell}^{\mathrm{OLS},W},
  &\text{otherwise, if the unrestricted fit is admissible},\\[1mm]
-\infty,&\text{if neither applies}.
\end{cases}
\label{eq:ols-critical-value}
\end{equation}
The OLS test rejects $H_0:\varrho=1$ against $H_1:\varrho<1$ when
$S_n^{\mathrm{OLS}}<q_{\tau,n,\ell}^{\mathrm{OLS}}$, and does not reject
otherwise. The conventions $q_{\tau,n,\ell}^{\mathrm{OLS}}=-\infty$ when
neither calibration is selected and $S_n^{\mathrm{OLS}}:=+\infty$ when the
OLS design is singular make both unavailable cases nonrejections.

The selection condition reflects the different behavior of constrained WLS
under the unit-root null and fixed stationary alternatives. Under $H_0$,
constrained WLS is consistent, so its fit is admissible with probability
tending to one and $\widetilde\delta_{n,\ell}^C\Pto\delta(\beta,\mu)>0$ on
that event; since $\eta_{n,\ell}\to0$, constrained calibration is selected
with probability tending to one. Under a fixed stationary alternative,
Proposition~\ref{prop:fixed-stationary} gives
\[
\widetilde\delta_{n,\ell}^C
=\Op(H_{n,\ell}^{-1})
=\op(\eta_{n,\ell}),
\]
on the constrained-admissibility event, so the selector switches to its
unrestricted calibration. We use unrestricted WLS because its nuisance pair is
consistent under both regimes.

The next theorem shows that the OLS test has asymptotically correct size
$\tau$ and is consistent against fixed stationary alternatives.

\begin{theorem}[Validity and consistency of the OLS test]
\label{thm:ols-test}
Under $H_0:\varrho=1$, assume that $(\beta,\mu)\in\Theta_0$, that
$(\beta',\mu')\mapsto
q_\tau^{\mathrm{OLS}}(\beta',\mu')$ is continuous at $(\beta,\mu)$, and that
\[
\PP\{\Psi_\varrho=q_\tau^{\mathrm{OLS}}(\beta,\mu)\}=0.
\]
Then
\[
q_{\tau,n,\ell}^{\mathrm{OLS}}
\Pto q_\tau^{\mathrm{OLS}}(\beta,\mu),
\qquad
\PP\{S_n^{\mathrm{OLS}}
<q_{\tau,n,\ell}^{\mathrm{OLS}}\}\longrightarrow\tau.
\]
Under every fixed stationary alternative $\varrho<1$ covered by
Proposition~\ref{prop:fixed-stationary} at which the OLS quantile map is
locally bounded,
\[
\PP\{S_n^{\mathrm{OLS}}
<q_{\tau,n,\ell}^{\mathrm{OLS}}\}\longrightarrow1.
\]
\end{theorem}

For computation, the squared-Bessel representation yields the scaling
identity
\[
q_\tau^{\mathrm{OLS}}(b,m)
=2(1+b)q_\tau^{\mathrm{base}}\{\delta(b,m)\},
\]
where $q_\tau^{\mathrm{base}}(d)$ denotes the lower $\tau$-quantile of
\[
\frac{
R_1^2-2(R_1+2)\displaystyle\int_0^1R_t\,dt
}{
4\left[
\displaystyle\int_0^1R_t^2\,dt
-\left(\displaystyle\int_0^1R_t\,dt\right)^2
\right]
},
\]
with $R$ defined by~\eqref{eq:squared-bessel-process} at dimension $\delta=d$.
This identity reduces critical-value
computation to a one-dimensional quantile map indexed by $\delta$.
We approximate this map by simulation and interpolation, as described in
Section~\ref{sec:mc-ols-test}; Theorem~\ref{thm:ols-test} is
stated for the exact limiting quantiles.

\section{Inference under the maintained unit-root model}
\label{sec:inference}

Throughout this section, we work under the unit-root assumption $\varrho=1$.
Section~\ref{sec:pretest-effect} also considers inference conditional on
nonrejection by the preceding test. The mixed convergence rates simplify
inference for smooth functions of
$\beta$ and $\mu$. When the derivative with respect to $\mu$ is nonzero
at the true parameter, uncertainty is dominated by estimation of $\mu$,
while estimation of $\beta$ has no first-order effect. Such functions inherit
the slow $\sqrt{\log n}$ convergence rate of the intercept estimator.

\subsection{Gaussian inference using plug-in variances}
\label{sec:gaussian-inference}

Let $h$ be a real-valued function that is continuously differentiable in a
neighborhood of $(\beta,\mu)$, with
$h_\mu(\beta,\mu)\ne0$. Because
$\sqrt n(\widetilde\beta_{n,\ell}^c-\beta)=\Op(1)$ and
$H_{n,\ell}/n\to0$, the contribution from estimating $\beta$ is
asymptotically negligible. The mixed-rate delta method gives
\[
\sqrt{H_{n,\ell}}
\{h(\widetilde\beta_{n,\ell}^c,\widetilde\mu_{n,\ell}^c)-h(\beta,\mu)\}
=h_\mu(\beta,\mu)\sqrt{H_{n,\ell}}
(\widetilde\mu_{n,\ell}^c-\mu)+\op(1)
\D h_\mu(\beta,\mu)Z.
\]
Functionals with $h_\mu(\beta,\mu)=0$ require a separate expansion.

For instance, for the long-run drift $g$ defined
in~\eqref{eq:long-run-drift}, the constrained WLS estimator is
$\widetilde g_{n,\ell}^{c}:=\widetilde\mu_{n,\ell}^{c}/
(1+\widetilde\beta_{n,\ell}^{c})$.
Taking $h(\beta,\mu)=\mu/(1+\beta)$ above gives
\[
\sqrt{H_{n,\ell}}(\widetilde g_{n,\ell}^{c}-g)
\D\mathcal N(0,\sigma_g^2),
\qquad
\sigma_g^2:=\frac{\sigma_\mu^2}{(1+\beta)^2}
=\frac{2(1-\beta)\beta\mu}{(1+\beta)^3}.
\]

The closed-form limiting variances yield plug-in estimators by substituting
$(\widetilde\beta_{n,\ell}^c,\widetilde\mu_{n,\ell}^c)$ for
$(\beta,\mu)$:
\[
\widehat\sigma_{\mu,n,\ell}^{2,\mathrm{PI}}
:=\frac{2(1-\widetilde\beta_{n,\ell}^c)
\widetilde\beta_{n,\ell}^c\widetilde\mu_{n,\ell}^c}
{1+\widetilde\beta_{n,\ell}^c},
\qquad
\widehat\sigma_{g,n,\ell}^{2,\mathrm{PI}}
:=\frac{\widehat\sigma_{\mu,n,\ell}^{2,\mathrm{PI}}}
{(1+\widetilde\beta_{n,\ell}^c)^2}.
\]
When the constrained fit is unique and its nuisance pair belongs to
$\Theta_0$, let $\widehat\sigma_{\mu,n,\ell}^{\mathrm{PI}}$ and
$\widehat\sigma_{g,n,\ell}^{\mathrm{PI}}$ denote the nonnegative square
roots of these estimators. For $\gamma\in(0,1)$, with $z_a$ the lower
$a$-quantile of the standard normal distribution, the plug-in intervals are
\begin{align*}
C_{\mu,n,\ell}^{\mathrm{PI}}(1-\gamma)
&:=\left[\widetilde\mu_{n,\ell}^c\mathbin{\pm}z_{1-\gamma/2}
\frac{\widehat\sigma_{\mu,n,\ell}^{\mathrm{PI}}}{\sqrt{H_{n,\ell}}}
\right]\cap[0,\infty),\\
C_{g,n,\ell}^{\mathrm{PI}}(1-\gamma)
&:=\left[\widetilde g_{n,\ell}^c\mathbin{\pm}z_{1-\gamma/2}
\frac{\widehat\sigma_{g,n,\ell}^{\mathrm{PI}}}{\sqrt{H_{n,\ell}}}
\right]\cap[0,\infty).
\end{align*}
Here $[a\mathbin{\pm}b]$ denotes $[a-b,a+b]$. If the constrained fit is
not admissible, these diagnostic intervals are recorded as unavailable.
Consistency of the variance estimators and Slutsky's lemma give asymptotic
coverage $1-\gamma$ under the maintained unit-root model.

Two features can make the plug-in variance unreliable at empirical sample sizes.
First, the predictable variance of the normalized intercept score depends on
the realized count path and includes an innovation-variance contribution,
whereas the plug-in formula replaces the path component by its limit and drops
the innovation term. These approximations are asymptotically valid, but the
logarithmic normalization allows appreciable finite-sample discrepancies.
Second, the plug-in variance is
proportional to $\widetilde\mu_{n,\ell}^c$, so a downward error in
$\widetilde\mu_{n,\ell}^c$ shrinks the interval center and its width together.
Section~\ref{sec:monte-carlo} confirms substantial undercoverage. Both problems
motivate estimating the variance from the constrained-WLS residuals, which
retain the realized variation in the weighted score.

\subsection{Residual-based Gaussian inference}
\label{sec:residual-based-inference}

In this subsection, we propose an alternative method for Gaussian inference
based on constrained-WLS residuals. Define the residuals and the corresponding
weighted residual-score (RS) variance estimator by
\begin{equation}
\widehat M_{k,\ell}
:=V_k+\widetilde\beta_{n,\ell}^cV_{k-1}
-\widetilde\mu_{n,\ell}^c,
\qquad
\widehat\sigma_{\mu,n,\ell}^{2,\mathrm{RS}}
:=\frac1{H_{n,\ell}}\sum_{k=1}^n
\frac{\widehat M_{k,\ell}^2}{(k+\ell)^2}.
\label{eq:residual-score-variance}
\end{equation}

\begin{lemma}[Residual-score variance consistency]
\label{lem:residual-score-variance}
Under the assumptions of Theorem~\ref{thm:main}, the estimator in
\eqref{eq:residual-score-variance} satisfies
\[
\widehat\sigma_{\mu,n,\ell}^{2,\mathrm{RS}}
\Pto\sigma_\mu^2.
\]
\end{lemma}

Put
\[
\widehat\sigma_{g,n,\ell}^{2,\mathrm{RS}}
:=\frac{\widehat\sigma_{\mu,n,\ell}^{2,\mathrm{RS}}}
{(1+\widetilde\beta_{n,\ell}^c)^2}.
\]
Let $\widehat\sigma_{\mu,n,\ell}^{\mathrm{RS}}$ and
$\widehat\sigma_{g,n,\ell}^{\mathrm{RS}}$ denote the nonnegative square
roots of the corresponding variance estimators.
The operational intervals are
\begin{align}
C_{\mu,n,\ell}^{\mathrm{RS}}(1-\gamma)
&:=\left[\widetilde\mu_{n,\ell}^c
\mathbin{\pm}z_{1-\gamma/2}
\frac{\widehat\sigma_{\mu,n,\ell}^{\mathrm{RS}}}
{\sqrt{H_{n,\ell}}}\right]\cap[0,\infty),
\label{eq:residual-score-mu-interval}\\
C_{g,n,\ell}^{\mathrm{RS}}(1-\gamma)
&:=\left[\widetilde g_{n,\ell}^c
\mathbin{\pm}z_{1-\gamma/2}
\frac{\widehat\sigma_{g,n,\ell}^{\mathrm{RS}}}
{\sqrt{H_{n,\ell}}}\right]\cap[0,\infty).
\label{eq:residual-score-g-interval}
\end{align}
An interval is available when the constrained fit is unique and finite and
the corresponding estimated variance is finite and positive; for $g$, the
denominator $1+\widetilde\beta_{n,\ell}^c$ must also be nonzero. Unavailable
intervals are represented by the empty set for coverage calculations.
Theorem~\ref{thm:constrained-wls},
Lemma~\ref{lem:residual-score-variance}, and the delta method imply
\[
\PP\{\mu\in C_{\mu,n,\ell}^{\mathrm{RS}}(1-\gamma)\}
\longrightarrow1-\gamma,
\qquad
\PP\{g\in C_{g,n,\ell}^{\mathrm{RS}}(1-\gamma)\}
\longrightarrow1-\gamma.
\]
Because both targets are positive, intersection with $[0,\infty)$ does not
alter these pointwise limits.

\subsection{Effect of OLS-test nonrejection}
\label{sec:pretest-effect}

Unit-root testing may affect both estimation and inference when the
constrained model is retained only after nonrejection. We show that, under
the unit-root null, this selection leaves the first-order distributions of
the estimators of $\mu$ and $g$ unchanged and preserves the asymptotic
coverage of their residual-based confidence intervals.
When the unit-root model is retained, estimation and inference use constrained
WLS at the same offset, regardless of which nuisance-calibration branch the
test selected.
For the operational selector in
\eqref{eq:ols-critical-value}, write
\[
A_{n,\ell}^{\mathrm{OLS}}
:=\{S_n^{\mathrm{OLS}}\geq q_{\tau,n,\ell}^{\mathrm{OLS}}\}
\]
for its nonrejection event.

\begin{coro}[First-order inference following nonrejection]
\label{cor:pretest-joint}
Under the assumptions of Theorem~\ref{thm:main} and the conditions
for the null assertion of Theorem~\ref{thm:ols-test}, the following
limits hold.
\[
\PP(A_{n,\ell}^{\mathrm{OLS}})\longrightarrow1-\tau,
\qquad
\mathcal L\!\left(
\sqrt{H_{n,\ell}}(\widetilde\mu_{n,\ell}^c-\mu)
\mid A_{n,\ell}^{\mathrm{OLS}}
\right)
\Longrightarrow\mathcal L(Z).
\]
Under the same conditions, for every $\gamma\in(0,1)$,
\begin{align*}
\PP\!\left\{
\mu\in C_{\mu,n,\ell}^{\mathrm{RS}}(1-\gamma)
\,\middle|\,A_{n,\ell}^{\mathrm{OLS}}\right\}
&\longrightarrow1-\gamma,\\
\PP\!\left\{
g\in C_{g,n,\ell}^{\mathrm{RS}}(1-\gamma)
\,\middle|\,A_{n,\ell}^{\mathrm{OLS}}\right\}
&\longrightarrow1-\gamma.
\end{align*}
\end{coro}

The result follows from the asymptotic independence of the intercept
estimation error $Z$ and the OLS test limit $\Psi_\varrho$. Since the
estimated critical value converges to a constant, conditioning on
nonrejection leaves the limiting distribution of $Z$ unchanged. The mixed
convergence rates extend this conclusion to $g$ and other smooth functions
with a nonzero derivative with respect to $\mu$, whose first-order
estimation errors depend only on $Z$. By contrast, $\Xi_w$ may depend on
$\Psi_\varrho$, so the same conclusion does not follow for $\beta$ or
$\alpha$.

\section{Monte Carlo evidence}
\label{sec:monte-carlo}

\subsection{Estimation}
\label{sec:finite-sample}

We compare OLS, unrestricted WLS (U-WLS), and unit-root-constrained WLS
(C-WLS) in transient, boundary, and recurrent unit-root designs. The parameter
settings are given in Table~\ref{tab:mc-main-designs}. For each design, we
simulate $20{,}000$ independent Poisson-INAR(2) trajectories from zero initial
counts at $n\in\{100,1000\}$. All methods use the same trajectories, with
shorter samples nested in longer ones; both WLS estimators use $\ell=0$.

\begin{table}[H]
\centering
\small
\caption{Parameter settings for the exact-unit-root simulation designs.}
\label{tab:mc-main-designs}
\begin{tabular}{lrrrr}
\toprule
Regime & $\alpha$ & $\beta$ & $\mu$ & $\delta$ \\
\midrule
Transient & 0.70 & 0.30 & 1 & 12.381 \\
Boundary & 0.70 & 0.30 & $21/130$ & 2 \\
Recurrent & 0.80 & 0.20 & 0.10 & 1.5 \\
\bottomrule
\end{tabular}
\end{table}

Table~\ref{tab:mc-estimation} reports root mean squared error (RMSE) for
estimates of the innovation mean $\mu$, the long-run drift $g=\mu/(1+\beta)$,
and the autoregressive parameter $\beta$. It also reports bias for estimates
of $\mu$ and $g$. Each method estimates $g$
by substituting its estimates of $\beta$ and $\mu$ into this ratio. Under
C-WLS, $\widetilde\alpha_{n,\ell}^c:=1-\widetilde\beta_{n,\ell}^c$, so the
RMSEs for $\alpha$ and $\beta$ are equal.

\begin{table}[H]
\centering
\small
\setlength{\tabcolsep}{3.5pt}
\renewcommand{\arraystretch}{0.96}
\caption{Finite-sample estimation under the exact unit root.}
\label{tab:mc-estimation}
\makeatletter
\begin{tabular}{lrlrrrrr}
\toprule
Regime & $n$ & Estimator & RMSE$(\beta)$ & RMSE$(\mu)$ &
Bias$(\mu)$ & RMSE$(g)$ & Bias$(g)$ \\
\midrule
Transient & 100 & OLS & 0.1109 & 0.8022 & 0.4940 & 0.6330 & 0.3911 \\
Transient & 100 & U-WLS & 0.1013 & 0.4719 & 0.1040 & 0.3598 & 0.0803 \\
Transient & 100 & C-WLS & 0.1010 & 0.3279 & 0.0017 & 0.2436 & 0.0004 \\
Transient & 1000 & OLS & 0.0351 & 0.7722 & 0.4845 & 0.5958 & 0.3737 \\
Transient & 1000 & U-WLS & 0.0314 & 0.3133 & 0.0544 & 0.2405 & 0.0417 \\
Transient & 1000 & C-WLS & 0.0314 & 0.2497 & 0.0016 & 0.1912 & 0.0011 \\
\addlinespace
Boundary & 100 & OLS & 0.1317 & 0.3728 & 0.2512 & 0.2998 & 0.2041 \\
Boundary & 100 & U-WLS & 0.1280 & 0.2130 & 0.0719 & 0.1742 & 0.0611 \\
Boundary & 100 & C-WLS & 0.1185 & 0.1295 & 0.0011 & 0.0989 & 0.0007 \\
Boundary & 1000 & OLS & 0.0396 & 0.3706 & 0.2615 & 0.2863 & 0.2023 \\
Boundary & 1000 & U-WLS & 0.0385 & 0.1363 & 0.0425 & 0.1052 & 0.0330 \\
Boundary & 1000 & C-WLS & 0.0383 & 0.1008 & 0.0009 & 0.0774 & 0.0007 \\
\addlinespace
Recurrent & 100 & OLS & 0.1429 & 0.2646 & 0.1666 & 0.2300 & 0.1472 \\
Recurrent & 100 & U-WLS & 0.1434 & 0.1648 & 0.0512 & 0.1469 & 0.0481 \\
Recurrent & 100 & C-WLS & 0.1300 & 0.0986 & -0.0001 & 0.0810 & -0.0004 \\
Recurrent & 1000 & OLS & 0.0420 & 0.2579 & 0.1776 & 0.2159 & 0.1488 \\
Recurrent & 1000 & U-WLS & 0.0426 & 0.1035 & 0.0311 & 0.0866 & 0.0261 \\
Recurrent & 1000 & C-WLS & 0.0425 & 0.0752 & -0.0004 & 0.0626 & -0.0004 \\

\bottomrule
\end{tabular}
\makeatother

\begin{minipage}{0.95\linewidth}
\footnotesize
\emph{Notes:} All three estimators are defined in all $20{,}000$ replications
of each reported design--sample-size combination. Monte Carlo standard errors
for the reported bias and RMSE estimates do not exceed $0.005$.
\end{minipage}
\end{table}

Table~\ref{tab:mc-estimation} shows little decline in OLS RMSE for $\mu$ and
$g$ as the sample grows, in line with the theoretical inconsistency of the
OLS intercept estimator under the unit root. U-WLS reduces the RMSE of $\mu$ by $38\%$--$63\%$
relative to OLS, with similar gains for $g$. Imposing the correct unit-root
restriction yields further reductions and small bias in both targets. The
methods have broadly comparable RMSE for $\beta$, with modest benefits from
weighting in the transient and boundary designs and a slight loss for U-WLS
in the recurrent design.

\subsection{OLS unit-root test}
\label{sec:mc-ols-test}

We compare the feasible OLS test with an infeasible oracle whose critical
values use the true nuisance parameters, using the three designs from the
estimation experiment at $n\in\{100,1000\}$. Size is assessed under $\varrho=1$,
and power under the fixed stationary alternative $\varrho=0.95$, holding
$\beta$ and $\mu$ unchanged and setting $\alpha=\varrho-\beta$. For each
design--hypothesis combination, we simulate $50{,}000$ independent trajectories
with Poisson innovations and zero initial counts. Both tests use the same
simulated trajectories and numerical quantile map. The feasible test uses
the reference offset $\ell=0$ and threshold $\eta_{n,0}=H_{n,0}^{-1/2}$.

Critical values are approximated using simulated squared-Bessel paths.\footnote{%
We estimate the 5\% base quantile at 41 dimension values between 0.1 and 100,
using 50,000 paths with 3,000 time steps per path. We interpolate the logarithm
of the quantile's absolute value linearly against log dimension. Below 0.1,
we continue in inverse proportion to dimension, matching the endpoint value;
above 100, we extrapolate using a least-squares line fitted to the last five
nodes in these logarithmic coordinates.}

\begin{table}[H]
\centering
\scriptsize
\setlength{\tabcolsep}{3pt}
\renewcommand{\arraystretch}{0.96}
\caption{Size and power of the OLS oracle and the proposed feasible OLS test.}
\label{tab:mc-ols-test}
\makeatletter
\begin{tabular}{lrrrrr}
\toprule
& & \multicolumn{2}{c}{OLS oracle (infeasible)} &
\multicolumn{2}{c}{Proposed feasible test} \\
\cmidrule(lr){3-4}\cmidrule(lr){5-6}
Regime & $n$ & $\varrho=1$ & $\varrho=.95$ &
$\varrho=1$ & $\varrho=.95$ \\
\midrule
Transient & 100 & 5.53 (0.10) & 76.29 (0.19) & 6.82 (0.11) & 49.50 (0.22) \\
 & 1000 & 5.14 (0.10) & 100.00 (0.00) & 5.02 (0.10) & 100.00 (0.00) \\
\addlinespace
Boundary & 100 & 4.72 (0.09) & 16.02 (0.16) & 5.52 (0.10) & 12.03 (0.15) \\
 & 1000 & 4.90 (0.10) & 99.67 (0.03) & 3.93 (0.09) & 80.89 (0.18) \\
\addlinespace
Recurrent & 100 & 4.57 (0.09) & 14.19 (0.16) & 6.85 (0.11) & 13.51 (0.15) \\
 & 1000 & 5.00 (0.10) & 99.29 (0.04) & 4.40 (0.09) & 81.53 (0.17) \\

\bottomrule
\end{tabular}
\makeatother

\begin{minipage}{0.96\linewidth}
\footnotesize
\emph{Notes:} Entries are rejection percentages at the nominal $5\%$ level,
with Monte Carlo standard errors in percentage points in parentheses.
The transient, boundary, and recurrent labels refer to the corresponding
unit-root designs.
Unavailable tests count as nonrejections. The standard errors exclude
numerical uncertainty in the critical values, and power comparisons are not
size adjusted.
\end{minipage}
\end{table}

The infeasible OLS oracle remains close to nominal size across all six null
settings, whereas the feasible test overrejects at $n=100$. At $n=1000$,
the feasible test's null rejection rates range from $3.93\%$ to $5.02\%$.
Under the stationary alternative, oracle power is higher in the transient
design at $n=100$ and approaches one in all three
designs at $n=1000$. Because $\mu$ varies across designs, these power
comparisons involve different long-run mean count levels $\mu/(1-\varrho)$.
The feasible test's power remains below the oracle's in the boundary and
recurrent designs at this sample size. Since both tests use the same statistic
and numerical quantile map, this gap reflects nuisance estimation and the
choice between constrained and unrestricted calibration.

To assess how the selection threshold affects finite-sample size and power,
we vary the multiplier $c_{\mathrm{sel}}$ in
$\eta_{n,0}=c_{\mathrm{sel}}H_{n,0}^{-1/2}$. Smaller multipliers favor
constrained-WLS calibration. In the boundary and recurrent designs, this
reduces both null rejection rates and power, revealing a size--power tradeoff.
The sensitivity results are reported in
Appendix~\ref{app:ols-test-tuning},
Table~\ref{tab:mc-wls-fallback-sensitivity}.

\subsection{Constrained inference}
\label{sec:mc-constrained-inference}

We evaluate confidence intervals under the maintained unit-root model using the same
three parameter designs. Alongside Poisson innovations, we consider
negative-binomial innovations with the same mean and twice the variance.
For each design and innovation law, we simulate $5{,}000$ trajectories at
$n\in\{100,1000\}$. Simulations start from zero
counts, with shorter samples nested in longer ones and offsets $\ell=0$ and
$\ell=5$ compared on the same trajectories.

We compare residual-score (RS) intervals with parameter plug-in intervals.
Coverage is evaluated under the maintained unit-root model, without first
applying the unit-root test.

\begin{table}[!htbp]
\centering
\scriptsize
\setlength{\tabcolsep}{2pt}
\renewcommand{\arraystretch}{0.94}
\caption{Coverage and mean length of nominal $95\%$ intervals under the unit root.}
\label{tab:mc-constrained-inference}
\makeatletter
\begin{tabular}{lrrrrcccc}
\toprule
& & & \multicolumn{2}{c}{Defined (\%)} &
\multicolumn{2}{c}{$\mu$} & \multicolumn{2}{c}{$g$} \\
\cmidrule(lr){4-5}\cmidrule(lr){6-7}\cmidrule(lr){8-9}
Regime & $n$ & $\ell$ & Plug-in & RS & Plug-in & RS & Plug-in & RS \\
\midrule
\multicolumn{9}{l}{\textit{Panel A: Poisson innovations}} \\
Transient & 100 & 0 & 99.7 & 100.0 & 84.6\,\% (0.944) & 90.0\,\% (1.118) & 85.4\,\% (0.722) & 90.5\,\% (0.862) \\
Transient & 100 & 5 & 99.8 & 100.0 & 95.0\,\% (1.248) & 93.6\,\% (1.147) & 95.4\,\% (0.957) & 94.2\,\% (0.887) \\
Transient & 1000 & 0 & 100.0 & 100.0 & 90.1\,\% (0.809) & 92.9\,\% (0.920) & 90.1\,\% (0.622) & 92.8\,\% (0.707) \\
Transient & 1000 & 5 & 100.0 & 100.0 & 96.1\,\% (0.970) & 94.9\,\% (0.914) & 96.2\,\% (0.746) & 94.9\,\% (0.704) \\
Boundary & 100 & 0 & 98.9 & 100.0 & 79.4\,\% (0.323) & 82.7\,\% (0.362) & 79.6\,\% (0.247) & 83.0\,\% (0.279) \\
Boundary & 100 & 5 & 99.0 & 100.0 & 90.5\,\% (0.394) & 92.0\,\% (0.372) & 90.8\,\% (0.302) & 91.9\,\% (0.288) \\
Boundary & 1000 & 0 & 100.0 & 100.0 & 87.2\,\% (0.299) & 88.6\,\% (0.325) & 87.0\,\% (0.229) & 88.3\,\% (0.250) \\
Boundary & 1000 & 5 & 100.0 & 100.0 & 95.1\,\% (0.344) & 94.3\,\% (0.328) & 95.1\,\% (0.264) & 94.1\,\% (0.252) \\
Recurrent & 100 & 0 & 94.8 & 100.0 & 70.7\,\% (0.212) & 78.0\,\% (0.238) & 70.0\,\% (0.172) & 78.0\,\% (0.198) \\
Recurrent & 100 & 5 & 95.0 & 100.0 & 83.0\,\% (0.261) & 88.9\,\% (0.248) & 83.2\,\% (0.213) & 89.0\,\% (0.208) \\
Recurrent & 1000 & 0 & 100.0 & 100.0 & 84.1\,\% (0.198) & 86.0\,\% (0.217) & 84.1\,\% (0.164) & 85.9\,\% (0.181) \\
Recurrent & 1000 & 5 & 100.0 & 100.0 & 93.3\,\% (0.229) & 93.2\,\% (0.221) & 93.4\,\% (0.190) & 93.5\,\% (0.184) \\
\addlinespace
\multicolumn{9}{l}{\textit{Panel B: Negative binomial innovations}} \\
Transient & 100 & 0 & 99.7 & 100.0 & 74.6\,\% (0.932) & 87.5\,\% (1.307) & 75.6\,\% (0.711) & 88.2\,\% (1.002) \\
Transient & 100 & 5 & 99.8 & 100.0 & 92.0\,\% (1.240) & 93.1\,\% (1.242) & 93.3\,\% (0.949) & 93.7\,\% (0.959) \\
Transient & 1000 & 0 & 100.0 & 100.0 & 82.6\,\% (0.801) & 91.4\,\% (1.050) & 83.1\,\% (0.615) & 91.5\,\% (0.807) \\
Transient & 1000 & 5 & 100.0 & 100.0 & 94.0\,\% (0.963) & 94.1\,\% (0.958) & 94.2\,\% (0.740) & 94.2\,\% (0.737) \\
Boundary & 100 & 0 & 99.1 & 100.0 & 70.5\,\% (0.311) & 77.0\,\% (0.382) & 70.9\,\% (0.237) & 77.3\,\% (0.292) \\
Boundary & 100 & 5 & 99.1 & 100.0 & 86.7\,\% (0.391) & 87.9\,\% (0.387) & 86.7\,\% (0.299) & 88.4\,\% (0.299) \\
Boundary & 1000 & 0 & 100.0 & 100.0 & 80.5\,\% (0.291) & 85.0\,\% (0.341) & 80.5\,\% (0.223) & 85.0\,\% (0.262) \\
Boundary & 1000 & 5 & 100.0 & 100.0 & 92.5\,\% (0.340) & 92.4\,\% (0.336) & 92.6\,\% (0.261) & 92.6\,\% (0.258) \\
Recurrent & 100 & 0 & 93.9 & 99.9 & 64.8\,\% (0.210) & 73.8\,\% (0.256) & 64.4\,\% (0.170) & 73.8\,\% (0.213) \\
Recurrent & 100 & 5 & 94.0 & 99.9 & 79.0\,\% (0.265) & 85.2\,\% (0.265) & 78.8\,\% (0.216) & 85.5\,\% (0.222) \\
Recurrent & 1000 & 0 & 100.0 & 100.0 & 78.8\,\% (0.195) & 82.6\,\% (0.230) & 79.1\,\% (0.162) & 82.8\,\% (0.192) \\
Recurrent & 1000 & 5 & 100.0 & 100.0 & 91.1\,\% (0.229) & 91.5\,\% (0.229) & 91.2\,\% (0.190) & 91.4\,\% (0.191) \\

\bottomrule
\end{tabular}
\makeatother

\begin{minipage}{0.97\linewidth}
\footnotesize
\emph{Notes:} Nominal coverage is $95\%$. Interval entries give unconditional
coverage in percent, with mean length in parentheses. The two defined-rate
columns give the percentage of samples for which each method's interval for
$\mu$ exists; the corresponding rates for $g$ are the same in these simulations.
Unavailable intervals count as noncoverage. Mean interval lengths are computed
over samples with defined intervals, after intersection with $[0,\infty)$.
The largest Monte Carlo standard error among the reported RS
coverage estimates is $0.62$ percentage point.
\end{minipage}
\end{table}

At $\ell=0$, residual-score intervals have higher coverage than plug-in
intervals in every reported design
(Table~\ref{tab:mc-constrained-inference}). Negative-binomial innovations
produce wider residual-score intervals and lower coverage than Poisson
innovations, with pronounced shortfalls in the boundary and recurrent
designs at $n=100$. Increasing the offset to $\ell=5$ raises residual-score coverage,
while the two methods vary in their closeness to nominal coverage across
designs.

We also examine a low-innovation design calibrated to the primary HFE fit
in Section~\ref{sec:flood-illustration}, with squared-Bessel dimension
$\delta\approx0.25$. We use the application's sample size, $n=123$, and a
longer sample of $n=200$, retaining the preceding experiment's innovation
laws, replication count, and paired offsets. These simulations start from
zero counts, whereas the observed HFE series starts from $(3,3)$, and do not
apply the unit-root test before constructing intervals. Intervals computed with the true
asymptotic variance (A-var) provide a benchmark for assessing the contribution
of variance estimation to undercoverage.

\begin{table}[H]
\centering
\scriptsize
\setlength{\tabcolsep}{3pt}
\caption{Application-calibrated inference stress design:
$\delta\approx0.250$.}
\label{tab:mc-application-stress}
\makeatletter
\begin{tabular}{rrrcccc}
\toprule
$n$ & $\ell$ & \shortstack{Fit valid\\(\%)} &
$\mu$: RS & $\mu$: A-var & $g$: RS & $g$: A-var \\
\midrule
\multicolumn{7}{l}{\textit{Panel A: Poisson innovations}} \\
123 & 0 & 91.9\,\% & 63.9\,\% (0.079) & 85.8\,\% & 65.1\,\% (0.056) & 85.6\,\% \\
123 & 5 & 91.9\,\% & 75.7\,\% (0.089) & 88.2\,\% & 76.3\,\% (0.063) & 88.0\,\% \\
200 & 0 & 98.3\,\% & 68.7\,\% (0.074) & 92.2\,\% & 69.2\,\% (0.051) & 92.0\,\% \\
200 & 5 & 98.3\,\% & 80.4\,\% (0.082) & 94.4\,\% & 81.0\,\% (0.057) & 94.1\,\% \\
\addlinespace
\multicolumn{7}{l}{\textit{Panel B: Negative binomial innovations}} \\
123 & 0 & 83.3\,\% & 56.3\,\% (0.086) & 77.7\,\% & 57.3\,\% (0.060) & 77.7\,\% \\
123 & 5 & 83.3\,\% & 67.0\,\% (0.096) & 79.4\,\% & 67.6\,\% (0.067) & 79.5\,\% \\
200 & 0 & 94.6\,\% & 61.4\,\% (0.076) & 88.9\,\% & 62.3\,\% (0.053) & 89.0\,\% \\
200 & 5 & 94.6\,\% & 72.6\,\% (0.085) & 90.7\,\% & 73.5\,\% (0.059) & 90.6\,\% \\

\bottomrule
\end{tabular}
\makeatother

\begin{minipage}{0.96\linewidth}
\footnotesize
\emph{Notes:} Nominal coverage is $95\%$. RS entries give unconditional
coverage in percent, with mean length in parentheses; A-var entries give
unconditional coverage using the true asymptotic variance. Conditional
coverage among valid fits is the A-var coverage divided by the fit-valid
percentage, times $100$. A valid fit has a unique, finite constrained
estimate; it need not belong to $\Theta_0$. Undefined fits
count as noncoverage. The largest coverage Monte Carlo standard error is
$0.71$ percentage point.
The parameter setting is
$(\alpha,\beta,\mu)\approx(0.538,0.462,0.0212)$. Simulations use
$(X_{-1},X_0)=(0,0)$; the observed HFE initial pair is $(3,3)$.
Each innovation law uses $5{,}000$ replications, with offsets $\ell=0$ and $5$
paired on each trajectory and shorter samples nested in longer ones.
\end{minipage}
\end{table}

Table~\ref{tab:mc-application-stress} shows substantial undercoverage at the
application's sample size, especially with overdispersed innovations. Using
the true asymptotic variance brings coverage close to nominal among samples
with valid fits; unavailable fits account for much of the remaining
unconditional shortfall for this benchmark.
Increasing the offset from $\ell=0$ to $\ell=5$ improves
residual-based coverage but does not eliminate the shortfall.

\section{Application to Canadian flood counts}
\label{sec:flood-illustration}

We analyze annual flood counts from two Canadian inventories. The Canadian
Disaster Database (CDD) records significant disasters, whereas Historical
Flood Events (HFE) compiles documented flood occurrences from multiple
sources. Their different inclusion criteria provide a useful comparison of
how inventory choice affects conclusions about persistence.

We count distinct events by calendar year, retaining years with no recorded
events, over 1902--2022 for CDD and 1900--2024 for HFE.\footnote{%
For CDD, we count distinct \texttt{EVENT\_ID} values with event category
\texttt{Disaster} and type \texttt{FL}. For HFE, we count distinct
\texttt{event\_id} values in \texttt{historical\_flood\_event\_0} for coastal
storm, frazil, freshet, and heavy rain. The year 2025 is excluded because it
is only partially covered in the source extract used here. Sources:
\href{https://www.publicsafety.gc.ca/cnt/rsrcs/cndn-dsstr-dtbs/index-en.aspx?wbdisable=true}
{Canadian Disaster Database} and
\href{https://open.canada.ca/data/en/dataset/fe83a604-aa5a-4e46-903c-685f8b0cc33c}
{Historical Flood Events}, including its
\href{https://download-telecharger.services.geo.ca/pub/nrcan_rncan/Floods_Inondation/historical_flood_event_downloads/HFE_ProductSpecifications.pdf}
{product specifications}. Changes in CDD coverage and reporting can affect
comparisons over time.}
HFE records substantially
more events and greater annual variation than CDD, while CDD contains more
zero-count years (Table~\ref{tab:flood-summary}).

\begin{table}[H]
\centering
\scriptsize
\setlength{\tabcolsep}{3.1pt}
\caption{Descriptive statistics for the two annual Canadian flood-count
series.}
\label{tab:flood-summary}
\makeatletter
\begin{tabular}{llrrrrrrrr}
\toprule
Data & Years & $N$ & $n$ & Events & Mean & SD & Median & Range & Zero years \\
\midrule
CDD & 1902--2022 & 121 & 119 & 344 & 2.843 & 2.652 & 2 & 0--11 & 28 \\
HFE & 1900--2024 & 125 & 123 & 1398 & 11.184 & 11.539 & 6 & 0--57 & 4 \\

\bottomrule
\end{tabular}
\makeatother

\begin{minipage}{0.96\linewidth}
\footnotesize
\emph{Notes:} Events are distinct source-specific identifiers. $N$ is the
number of zero-completed annual observations, and $n=N-2$ is the number of
fitted transitions after conditioning on the first two counts. SD is the
sample standard deviation of the annual counts.
\end{minipage}
\end{table}

We condition on the first two annual counts and apply the OLS unit-root test
at the $5\%$ level. Following \citet{lu2026intercept}, we use $\ell=0$ as the
primary offset and $\ell=2$ to assess sensitivity to the time origin.
Table~\ref{tab:flood-unit-root} reports the selected calibration branch and
its fitted nuisance pair.

\begin{table}[H]
\centering
\footnotesize
\setlength{\tabcolsep}{2.5pt}
\caption{OLS unit-root tests at the 5\% level.}
\label{tab:flood-unit-root}
\makeatletter
\begin{tabular}{lrlrrrrl}
\toprule
Data & $\ell$ & Calibration & $b$ & $m$ & OLS statistic & Critical value & Decision \\
\midrule
CDD & 0 & C-WLS & 0.557 & 0.179 & -33.912 & -23.145 & Reject \\
CDD & 2 & C-WLS & 0.488 & 0.076 & -33.912 & -40.862 & Do not reject \\
HFE & 0 & U-WLS & 0.423 & 0.287 & -12.041 & -16.442 & Do not reject \\
HFE & 2 & C-WLS & 0.453 & 0.099 & -12.041 & -33.527 & Do not reject \\

\bottomrule
\end{tabular}
\makeatother

\begin{minipage}{0.96\linewidth}
\footnotesize
\emph{Notes:} The unit-root null is rejected when the OLS statistic is below
the selected critical value. C-WLS and U-WLS denote constrained and
unrestricted WLS calibration, respectively; $(b,m)$ is the selected fitted
nuisance pair used to evaluate $q_{.05}^{\mathrm{OLS}}(b,m)$.
\end{minipage}
\end{table}

For CDD, the test rejects the unit-root null at $\ell=0$ but not at
$\ell=2$. For HFE, it does not reject at either offset. We therefore report
constrained estimation and inference for HFE under the maintained unit-root
model.

At the primary offset the estimated innovation mean is $0.021$, with a
nominal $95\%$ residual-score interval of $[0,0.567]$
(Table~\ref{tab:hfe-residual-inference}); the interval for the long-run
drift is correspondingly wide. The application-calibrated simulations in
Section~\ref{sec:mc-constrained-inference} show substantial undercoverage at
$n=123$. They use zero initial counts and do not condition on test
nonrejection, so they do not estimate coverage for the full empirical
procedure. We therefore report the intervals as nominal asymptotic intervals;
their finite-sample $95\%$ coverage is not established for this application.

\begin{table}[H]
\centering
\small
\setlength{\tabcolsep}{3pt}
\caption{Residual-score Gaussian inference for HFE following unit-root
nonrejection.}
\label{tab:hfe-residual-inference}
\makeatletter
\begin{tabular}{rrrrcc}
\toprule
$\ell$ & $\widetilde\beta^c$ & $\widetilde\mu^c$ &
$\widetilde g^c$ & $95\%$ interval for $\mu$ &
$95\%$ interval for $g$ \\
\midrule
0 & 0.462 & 0.021 & 0.015 & $[0.000, 0.567]$ & $[0.000, 0.388]$ \\
2 & 0.453 & 0.099 & 0.068 & $[0.000, 0.704]$ & $[0.000, 0.485]$ \\

\bottomrule
\end{tabular}
\makeatother

\begin{minipage}{0.96\linewidth}
\footnotesize
\emph{Notes:} The constrained fits condition on the observed HFE initial pair
$(X_{-1},X_0)=(3,3)$. Intervals use
\eqref{eq:residual-score-mu-interval}--\eqref{eq:residual-score-g-interval}
and are intersected with $[0,\infty)$. The primary specification is $\ell=0$;
$\ell=2$ was fixed in advance as the application sensitivity check.
\end{minipage}
\end{table}

\section{Conclusion}
We have shown that, unlike OLS, inverse-time WLS estimates the intercept of a
unit-root INAR(2) process consistently, with a Gaussian limit independent of
the autoregressive limits. Under the unit-root model, this independence keeps
residual-score intervals for the innovation mean and long-run drift pointwise
asymptotically valid after nonrejection by our WLS-calibrated OLS unit-root
test. In simulations, WLS also sharply reduces intercept RMSE relative to OLS.

Two limitations motivate further work: the slow $\sqrt{\log n}$
convergence of the intercept estimator and the substantial finite-sample
undercoverage of residual-score intervals, particularly in the design
calibrated to the application. State weighting offers a possible way
forward. For first-order integrated Galton--Watson processes,
\citet{lu2026stateweights} shows that state weighting yields faster
intercept convergence under recurrence and a smaller asymptotic variance
under transience than inverse-time weighting. Extending this approach
to unit-root INAR(2) models and assessing whether it also improves
finite-sample coverage would be a natural continuation of the present study.

\section*{Declaration of generative AI and AI-assisted technologies in the manuscript preparation process}
During the preparation of this manuscript, the authors used OpenAI Codex to
assist with literature searches, review of mathematical arguments, language
and exposition, and LaTeX formatting. The authors reviewed and verified the
resulting content and take full responsibility for the manuscript.


\appendix
\numberwithin{equation}{section}
\numberwithin{table}{section}
\numberwithin{figure}{section}

\section{Preliminary results}

Throughout the appendix, the unit-root assumption is in force except where
a stationary alternative is explicitly considered.
The symbol $C\in(0,\infty)$ denotes a generic constant that may
change from line to line. Besides the model parameters, it may depend on the
fixed initial condition and on the fixed offset $\ell$ when these are present.
We also write $c_\beta:=2\beta/(1+\beta)$.

Until the fixed-initial-condition and fixed-offset transfer in
Proposition~\ref{prop:joint-convergence}, we work under
$X_{-1}=X_0=0$ and $\ell=0$, and abbreviate
$H_{n,0}$, $A_{n,0}^{(w)}$, and $d_{n,0}^{(w)}$ by
$H_n$, $A_n^{(w)}$, and $d_n^{(w)}$. The transfer proposition extends the joint
design-and-score limit to every fixed initial condition and every fixed
$\ell\geq0$.

We use the scaled step process
\[
\mathcal X_t^{(n)}:=n^{-1}X_{\lfloor nt\rfloor},
\qquad t\in[0,1].
\]
For the appendix, define $U_k:=X_k+\beta X_{k-1}$.
We also use the step processes in equation~(5.6) of BIP:
\[
\mathcal M_t^{(n)}:=\sum_{k=1}^{\lfloor nt\rfloor} n^{-1}M_k,
\qquad
\mathcal N_t^{(n)}:=\sum_{k=1}^{\lfloor nt\rfloor} n^{-2}M_kU_{k-1},
\qquad
\mathcal P_t^{(n)}:=\sum_{k=1}^{\lfloor nt\rfloor} n^{-3/2}M_kV_{k-1}.
\]
That equation gives directly
\begin{equation}
\bigl(\mathcal X^{(n)},\mathcal M^{(n)},\mathcal N^{(n)},\mathcal P^{(n)}\bigr)
\D
\bigl(\mathcal X,\mathcal M,\mathcal N,\mathcal P\bigr)
\quad\text{in }D([0,1],\R^4),
\label{eq:barczy-process-limit}
\end{equation}
where
\[
\mathcal M_t=(1+\beta)\mathcal X_t-\mu t,
\qquad
\mathcal N_t=(1+\beta)\sqrt{2\alpha\beta}\int_0^t \mathcal X_s^{3/2}\,dW_s,
\qquad
\mathcal P_t=2\beta\sqrt{\frac{\alpha}{1+\beta}}\int_0^t \mathcal X_s\,d\widetilde W_s.
\]

We use the following decomposition throughout the appendix. Let $m_k:=\E(X_k)$ and define the centered processes
\[
\widetilde U_k:=U_k-\E(U_k)=\sum_{j=1}^kM_j,
\qquad
\widetilde V_k:=V_k-\E(V_k)=\sum_{j=1}^k(-\beta)^{k-j}M_j.
\]
Since $X_k=(U_k+\beta V_k)/(1+\beta)$,
\begin{equation}
X_k
=
m_k+\frac{\widetilde U_k+\beta\widetilde V_k}{1+\beta},
\qquad
m_k
=
\frac{\mu}{1+\beta}\,k
+\frac{\mu\beta}{(1+\beta)^2}\bigl(1-(-\beta)^k\bigr).
\label{eq:centered-decomposition}
\end{equation}
Indeed, the canonical recursions give $U_k=U_{k-1}+\mu+M_k$ and $V_k=-\beta V_{k-1}+\mu+M_k$; centering and iteration give the displayed formulas, while summing
\[
m_k-m_{k-1}
=
\frac{\mu}{1+\beta}\bigl(1-(-\beta)^k\bigr)
\]
gives the expression for $m_k$. We also repeatedly use the moment bounds from Corollary~A.4 of BIP:
\begin{equation}
\begin{gathered}
\E(X_k)=\mathrm O(k),\qquad
\E(X_k^2)=\mathrm O(k^2),\qquad
\E(M_k^2)=\mathrm O(k),\qquad
\E(M_k^4)=\mathrm O(k^2),\\
\E(U_k^4)=\mathrm O(k^4),\qquad
\E(V_k^2)=\mathrm O(k),\qquad
\E(V_k^4)=\mathrm O(k^2).
\end{gathered}
\label{eq:moment-bounds}
\end{equation}

\needspace{10\baselineskip}
\section{Auxiliary weighted-sum estimates}

The following lemma collects the initial-segment estimates needed to control
the truncation error in Proposition~\ref{prop:joint-convergence-base}.

\begin{lemma}
\label{lem:initial-segment-bounds}
Fix $\zeta\in(0,1)$ and let $j_n(\zeta):=\lfloor n\zeta\rfloor+1$. Uniformly in $n$,
\begin{align}
\E\left[\frac1n\sum_{k=1}^{j_n(\zeta)-1}\frac{X_{k-1}}{k}\right]
&\le C\zeta,
&
\E\left[\frac1{n^2}\sum_{k=1}^{j_n(\zeta)-1}\frac{X_{k-1}^2}{k}\right]
&\le C\zeta^2,
\label{eq:initial-segment-design}\\
\E\left(\frac1n\sum_{k=1}^{j_n(\zeta)-1}\frac{M_kU_{k-1}}{k}\right)^2
&\le C\zeta^2,
&
\E\left(\frac1{\sqrt n}\sum_{k=1}^{j_n(\zeta)-1}\frac{M_kV_{k-1}}{k}\right)^2
&\le C\zeta.
\label{eq:initial-segment-score}
\end{align}
Moreover,
\begin{equation}
\E\left(\frac1n\sum_{k=1}^n\frac{M_kV_{k-1}}{k}\right)^2
\le \frac Cn.
\label{eq:negligible-V-score}
\end{equation}
\end{lemma}

\begin{proof}
The first two estimates follow immediately from \eqref{eq:moment-bounds}. For the score terms, Cauchy--Schwarz and \eqref{eq:moment-bounds} give
\[
\E(M_k^2U_{k-1}^2)=\mathrm O(k^3),
\qquad
\E(M_k^2V_{k-1}^2)=\mathrm O(k^2).
\]
The weighted score summands are martingale differences and hence orthogonal. Therefore
\[
\E\left(\frac1n\sum_{k=1}^{j_n(\zeta)-1}\frac{M_kU_{k-1}}{k}\right)^2
\le
\frac C{n^2}\sum_{k=1}^{j_n(\zeta)-1}k
\le C\zeta^2,
\]
and the same calculation for the $V$ score gives the second estimate in \eqref{eq:initial-segment-score}. Taking the $V$ sum up to $n$ with normalization $n^{-1}$ gives \eqref{eq:negligible-V-score}.
\end{proof}

The next lemma uses \eqref{eq:centered-decomposition} to control the weighted sums needed for the predictable quadratic variation of the intercept score.

\begin{lemma}
\label{lem:weighted-V-k2}
For $r\in\{1,2\}$,
\begin{equation}
\frac1{H_n}\sum_{k=1}^n \frac{X_{k-r}}{k^2}
\Ltwo
\frac{\mu}{1+\beta}.
\label{eq:X-over-k2}
\end{equation}
\end{lemma}

\begin{proof}
Fix $r\in\{1,2\}$. By \eqref{eq:centered-decomposition},
\[
\sum_{k=1}^n\frac{X_{k-r}-m_{k-r}}{k^2}
=
\frac1{1+\beta}
\left(
\sum_{k=1}^n\frac{\widetilde U_{k-r}}{k^2}
+\beta\sum_{k=1}^n\frac{\widetilde V_{k-r}}{k^2}
\right).
\]
After interchanging the sums, the coefficients of $M_j$ in the two terms on the right are bounded respectively by
\[
\sum_{k=j+r}^\infty\frac1{k^2}\le\frac Cj,
\qquad
\sum_{k=j+r}^\infty\frac{\beta^{k-r-j}}{k^2}\le\frac C{j^2}.
\]
Orthogonality and \eqref{eq:moment-bounds} therefore give
\[
\E\left(\sum_{k=1}^n\frac{X_{k-r}-m_{k-r}}{k^2}\right)^2
\le CH_n,
\]
so the centered part divided by $H_n$ converges to zero in $L^2$. Finally, \eqref{eq:centered-decomposition} yields
\[
\sum_{k=1}^n\frac{m_{k-r}}{k^2}
=
\frac{\mu}{1+\beta}H_n+\mathrm O(1),
\]
where the finitely many boundary terms are absorbed into $\mathrm O(1)$. This proves \eqref{eq:X-over-k2}.
\end{proof}

The following lemma controls the weighted sums involving $V_{k-1}$: it
provides the probability orders of $\sum k^{-1}V_{k-1}$ and
$\sum k^{-1}X_{k-1}V_{k-1}$ (used for the off-diagonal entries in
Proposition~\ref{prop:joint-convergence-base}), and relates
$\sum k^{-1}V_{k-1}^2$ to $\sum k^{-1}X_{k-1}$ (used for the $(2,2)$
diagonal entry).

\begin{lemma}
\label{lem:weighted-V-sums}
Under the zero-start convention $X_{-1}=X_0=0$, as $n\to\infty$,
\begin{equation}
\sum_{k=1}^n \frac{V_{k-1}}{k}=\Op(H_n)=\Op(\log n),
\qquad
\sum_{k=1}^n \frac{X_{k-1}V_{k-1}}{k}=\Op(n),
\label{eq:XV-bounds}
\end{equation}
and
\begin{equation}
\frac1n\sum_{k=1}^n \frac{V_{k-1}^2}{k}
-
\frac{2\beta}{1+\beta}\frac1n\sum_{k=1}^n \frac{X_{k-1}}{k}
\xrightarrow{L^1}0.
\label{eq:V2-weighted}
\end{equation}
\end{lemma}

\begin{proof}
First, using the telescoping identity noted by BIP
\citep[p.~876]{barczy2014asymptotic}, summation by parts gives
\[
\sum_{k=1}^n\frac{V_{k-1}}k
=\frac{X_{n-1}}n
+\sum_{k=1}^{n-1}\frac{X_{k-1}}{k(k+1)}.
\]
Since $\E(X_k)=\mathrm O(k)$, the expectation of the nonnegative right-hand
side is $\mathrm O(H_n)$. Markov's inequality therefore gives the first
statement in \eqref{eq:XV-bounds}.

Next, using $X_{k-1}V_{k-1}=\frac12(X_{k-1}^2-X_{k-2}^2+V_{k-1}^2)$,
\[
\sum_{k=1}^n \frac{X_{k-1}V_{k-1}}{k}
=
\frac12\sum_{k=1}^n \frac{X_{k-1}^2-X_{k-2}^2}{k}
+\frac12\sum_{k=1}^n \frac{V_{k-1}^2}{k}.
\]
By summation by parts,
\[
\sum_{k=1}^n \frac{X_{k-1}^2-X_{k-2}^2}{k}
=
\frac{X_{n-1}^2}{n}+\sum_{k=1}^{n-1}\frac{X_{k-1}^2}{k(k+1)}.
\]
Since $\E(X_k^2)=\mathrm O(k^2)$, this term is $\Op(n)$. Also $\E(V_k^2)=\mathrm O(k)$ implies $\sum V_{k-1}^2/k=\Op(n)$. Hence the second statement in \eqref{eq:XV-bounds}.

For \eqref{eq:V2-weighted}, put
$\Delta_k:=V_k^2-\E(V_k^2\mid\mathcal F_{k-1})$ and
$\mathcal R_m:=\sum_{k=1}^m(V_{k-1}^2-c_\beta X_{k-1})$.
Adapting BIP's decomposition
\citep[p.~876, equation~(5.3)]{barczy2014asymptotic}, using
$(1-\beta^2)c_\beta=2\alpha\beta$ and $V_0=X_{-1}=0$, gives
\[
 (1-\beta^2)\mathcal R_m
 =\sum_{k=1}^m\Delta_k-V_m^2+c_0m+c_1X_{m-1},
\]
where $c_0=\Var(\varepsilon_1)+\mu^2$ and
$c_1=-(2\beta\mu+\alpha\beta)$.
Martingale orthogonality and BIP's Corollary~A.4 give
$\E(\Delta_k^2)\leq\E(V_k^4)=O(k^2)$ and
$\E(V_m^2)+\E(X_{m-1})=O(m)$, hence $\E|\mathcal R_m|\leq Cm^{3/2}$.
Abel summation therefore yields
\[
 \E\left|\frac1n\sum_{k=1}^n
       \frac{V_{k-1}^2-c_\beta X_{k-1}}k\right|
 \leq \frac{\E|\mathcal R_n|}{n^2}
       +\frac1n\sum_{k=1}^{n-1}\frac{\E|\mathcal R_k|}{k(k+1)}
 \leq Cn^{-1/2}\longrightarrow0,
\]
which proves \eqref{eq:V2-weighted}.
\end{proof}

\section{Joint limit of the weighted design matrix and score vector}

The first two weighted score coordinates are treated by summation by parts on
$[\zeta,1]$ and the process convergence~\eqref{eq:barczy-process-limit}. For
the intercept score $Z_n:=H_n^{-1/2}\sum_{k=1}^n M_k/k$, an initial block of
$o(n)$ observations accounts for asymptotically all of its variance. We apply
a martingale central limit theorem to this block. Removing the families
present at its endpoint leaves an independent process with the same
diffusion limit, which establishes the required asymptotic independence.

For the decompositions below, we use an individual-labelled branching
realization of the INAR(2) process. Let $\mathcal I_k$ be the set of
individuals counted by $X_k$. The immigrants at time $k$ receive labels
$\iota_{k,i}$, $1\leq i\leq\varepsilon_k$. Each individual
$u\in\mathcal I_{k-1}$ receives an independent Bernoulli-$\alpha$ mark for a
child at time $k$, and each $u\in\mathcal I_{k-2}$ receives an independent
Bernoulli-$\beta$ mark for a child at time $k$. A child label records its
parent and birth time, so immigrant and offspring labels are disjoint and
retain their complete genealogies. Thus
\[
 \mathcal I_k
 =\{\iota_{k,i}:1\leq i\leq\varepsilon_k\}
 \mathbin{\dot\cup}
 \{(u,k,\alpha):u\in\mathcal I_{k-1},\ \xi_{k,u}^{(\alpha)}=1\}
 \mathbin{\dot\cup}
 \{(u,k,\beta):u\in\mathcal I_{k-2},\ \xi_{k,u}^{(\beta)}=1\},
\]
and $X_k=|\mathcal I_k|$ has the required INAR(2) law. Initial individuals,
when present, receive distinct pre-sample root labels with root times $-1$ and
$0$, respectively. Let
$\mathcal F_k^{\mathrm{br}}$ be
the filtration generated by the initial labels, the immigrant counts through
time $k$, and all reproduction marks used to form generations through time
$k$. Then $\mathcal F_k\subseteq\mathcal F_k^{\mathrm{br}}$, and $(M_k)$
remains a martingale-difference sequence with respect to
$(\mathcal F_k^{\mathrm{br}})$. Counts may be partitioned pathwise by disjoint sets of root
labels; the resulting component residuals are conditionally centered, and
components with disjoint roots use disjoint reproduction marks.

\Needspace{10\baselineskip}
The blocking argument uses the following survival bound for families
already present at the blocking time.

\begin{lemma}[Survival of old families]
\label{lem:old-family-survival}
Suppose that $\alpha+\beta=1$, with $\alpha,\beta\in(0,1)$.
Let $s\geq0$ be an integer. Conditionally on
$\mathcal F_s^{\mathrm{br}}$, let $D$ be the population
generated after time $s$ by $D_s+D_{s-1}=:K_s$ existing individuals, with
no subsequent immigration. There is a constant $C$ such that, for every integer $r\geq s+3$,
\[
 \PP\{D_k>0\text{ for some }k\geq r\mid\mathcal F_s^{\mathrm{br}}\}
 \leq \frac{C K_s}{r-s}.
\]
In particular, a family started from finitely many individuals becomes
extinct in finite calendar time almost surely.
\end{lemma}
\begin{proof}
An individual has independent Bernoulli-$\alpha$ and Bernoulli-$\beta$
children, born one and two time units later. If birth times are suppressed,
its family is a Galton--Watson tree with offspring generating function
$f(z)=(1-\alpha+\alpha z)(1-\beta+\beta z)$.
Writing $q_j$ for survival to genealogical generation $j$, we have $q_0=1$
and
\[
 q_{j+1}=1-f(1-q_j)=q_j-\alpha\beta q_j^2,
 \qquad
 \frac1{q_{j+1}}\geq\frac1{q_j}+\alpha\beta.
\]
Consequently $q_j\leq(1+\alpha\beta j)^{-1}$.
Cut all family lines at time $s$. The unused reproduction of a root born at
$s$ has the full offspring law above. A root born at $s-1$ has only its
Bernoulli-$\beta$ reproduction remaining, which is dominated by the full law.
Future descendants are therefore covered by $K_s$ such trees.
Since every parent--child edge takes at most two calendar time units,
a birth at or after $r$ requires genealogical depth at least
$\lfloor(r-s)/2\rfloor$ in one of these trees. The conditional union bound
proves the inequality. Letting $r\to\infty$ gives the extinction assertion.
\end{proof}

The next lemma supplies the joint limit needed below.

\begin{lemma}
\label{lem:joint-process-Rn}
In $D([0,1],\R^3)\times\R$,
\begin{equation}
\bigl(\mathcal X^{(n)},\mathcal N^{(n)},\mathcal P^{(n)},Z_n\bigr)
\D
\bigl(\mathcal X,\mathcal N,\mathcal P,Z\bigr),
\label{eq:joint-process-Rn}
\end{equation}
where $Z\sim\mathcal N(0,\sigma_\mu^2)$. In the Brownian representation of
the BIP limit, $Z$ is independent of $(\mathcal X,W,\widetilde W)$.
\end{lemma}

\begin{proof}
For sufficiently large $n$, set
\[
 s_n=\left\lfloor\frac{n}{(\log n)^2}\right\rfloor,
 \qquad r_n=\left\lfloor\frac{n}{\log n}\right\rfloor,
 \qquad Z_n^*=H_n^{-1/2}\sum_{k=1}^{s_n}\frac{M_k}{k}.
\]
Then $H_{s_n}/H_n\to1$. Martingale orthogonality and
\eqref{eq:moment-bounds} give
\[
 \E|Z_n-Z_n^*|^2
 \leq\frac C{H_n}\sum_{k=s_n+1}^n\frac1k
 =O\!\left(\frac{\log\log n}{\log n}\right)=o(1).
\]
Lemma~\ref{lem:weighted-V-k2}, applied with sample size $s_n$, together
with $H_{s_n}/H_n\to1$, gives, for $r=1,2$,
\[
 \frac1{H_n}\sum_{k=1}^{s_n}\frac{X_{k-r}}{k^2}
 =\frac{H_{s_n}}{H_n}
  \left\{\frac1{H_{s_n}}\sum_{k=1}^{s_n}\frac{X_{k-r}}{k^2}\right\}
 \xrightarrow{L^2}\frac\mu{1+\beta}.
\]
Also,
\[
 \frac{\Var(\varepsilon_1)}{H_n}\sum_{k=1}^{s_n}k^{-2}
 =O(H_n^{-1})\longrightarrow0.
\]
The conditional variance identity
$\E(M_k^2\mid\mathcal F_{k-1}^{\mathrm{br}})
=\alpha\beta(X_{k-1}+X_{k-2})+\Var(\varepsilon_1)$
therefore yields
\[
 \frac1{H_n}\sum_{k=1}^{s_n}
 \frac{\E(M_k^2\mid\mathcal F_{k-1}^{\mathrm{br}})}{k^2}
 \Pto\frac{2\alpha\beta\mu}{1+\beta}=\sigma_\mu^2.
\]
For every $a>0$, the expectation of the conditional Lindeberg sum is at most
\[
 \frac1{a^2H_n^2}\sum_{k=1}^{s_n}\frac{\E(M_k^4)}{k^4}
 \leq \frac C{a^2H_n^2}\sum_{k=1}^\infty k^{-2}\longrightarrow0.
\]
Thus the martingale CLT
\citep[Corollary~3.1, pp.~58--59]{hall1980martingale} gives
$Z_n^*\D\mathcal N(0,\sigma_\mu^2)$.

Let $\overline X_k$ consist of the immigrants arriving after $s_n$ and their
descendants in the labelled construction, and write $D_k=X_k-\overline X_k$.
The entire barred process is independent of $\mathcal F_{s_n}^{\mathrm{br}}$.
Set $\overline X_k=0$ for $k\leq s_n$ and define
\[
 \overline U_k=\overline X_k+\beta\overline X_{k-1},
 \qquad \overline V_k=\overline X_k-\overline X_{k-1},
 \qquad \overline M_k=\overline X_k-\alpha\overline X_{k-1}
       -\beta\overline X_{k-2}-\mu\mathbf1_{\{k>s_n\}}.
\]
Define $\overline{\mathcal X}^{(n)},\overline{\mathcal N}^{(n)},
\overline{\mathcal P}^{(n)}$ from these variables with the original
normalizations.

Since $\E(X_{s_n}+X_{s_n-1})=O(s_n+1)$, the survival lemma gives
\[
 \PP(E_n^c)\leq C\frac{s_n+1}{r_n-s_n-2}\longrightarrow0,
 \qquad E_n:=\{D_k=0\text{ for every }k\geq r_n-2\}.
\]
On $E_n$ the two population paths agree after $r_n-2$, and the increments
of both score processes agree for $k\geq r_n$. Their cumulative scores can
still differ by the constant accumulated before $r_n$.

The process $U_k$ is a nonnegative submartingale. Doob's inequality and
\eqref{eq:moment-bounds}, together with martingale orthogonality, yield
\[
 \E\max_{j\leq r}X_j^2\leq Cr^2,
 \qquad
 \E\max_{j\leq r}\left|\sum_{k=1}^j M_kU_{k-1}\right|^2\leq Cr^4,
 \qquad
 \E\max_{j\leq r}\left|\sum_{k=1}^j M_kV_{k-1}\right|^2\leq Cr^3.
\]
For the last two bounds use
$\E(M_k^2U_{k-1}^2)\leq Ck^3$ and
$\E(M_k^2V_{k-1}^2)\leq Ck^2$.
The same bounds hold for the barred process: after $s_n$ it has the law of
a fresh zero-start process. With
$T_n=(\mathcal X^{(n)},\mathcal N^{(n)},\mathcal P^{(n)})$ and
$\overline T_n$ defined analogously, Markov's inequality therefore gives,
for every $a>0$,
\[
 \PP(\|T_n-\overline T_n\|_\infty>a)
 \leq\PP(E_n^c)+C_a\left\{
 (r_n/n)^2+(r_n/n)^4+(r_n/n)^3\right\}\longrightarrow0.
\]
Here any fixed norm on $\mathbb R^3$ may be used.

The BIP process limit \eqref{eq:barczy-process-limit} and $s_n/n\to0$
give $\overline T_n\D(\mathcal X,\mathcal N,\mathcal P)$: its law is that
of a fresh zero-start process delayed by $s_n/n$, and the limit is continuous.
For each $n$, $\overline T_n$ and $Z_n^*$ are independent. Product convergence,
the uniform approximation above, and $Z_n-Z_n^*\Pto0$ now prove
\eqref{eq:joint-process-Rn}. Realize the BIP Brownian representation on a
product space with an independent $Z$; this supplies the asserted
independence from $(\mathcal X,W,\widetilde W)$.
\end{proof}

The process coordinates of the limit in \eqref{eq:joint-process-Rn} have continuous paths. By the Skorokhod representation theorem, we may realize this joint convergence so that it holds almost surely in $D([0,1],\R^3)\times\R$. For the process coordinates, almost sure $J_1$ convergence is then equivalent to uniform convergence on $[0,1]$. Hence the fixed-$\zeta$ maps used below are almost surely continuous.

\Needspace{12\baselineskip}
The next proposition obtains the design and score limits jointly.

\begin{proposition}
\label{prop:joint-convergence-base}
Under the reference conditions $X_{-1}=X_0=0$ and $\ell=0$, let
\[
\mathsf{D}_n:=\diag\bigl(n^{-1},n^{-1/2},H_n^{-1/2}\bigr),
\qquad
\widetilde A_n^{(w)}:=\mathsf{D}_nA_n^{(w)}\mathsf{D}_n,
\qquad
\widetilde d_n^{(w)}:=\mathsf{D}_nd_n^{(w)}.
\]
Define
\[
\begin{aligned}
\widetilde A^{(w)}
&:=
\begin{pmatrix}
I_2 & 0 & 0\\
0 & \frac{2\beta}{1+\beta}I_1 & 0\\
0 & 0 & 1
\end{pmatrix},
&
\widetilde d^{(w)}
&:=
\begin{pmatrix}
\sqrt{2\alpha\beta}\int_0^1 t^{-1}\mathcal X_t^{3/2}\,dW_t\\[2mm]
-2\beta\sqrt{\frac{\alpha}{1+\beta}}\int_0^1 t^{-1}\mathcal X_t\,d\widetilde W_t\\[2mm]
Z
\end{pmatrix}.
\end{aligned}
\]
where $Z\sim\mathcal N(0,\sigma_\mu^2)$ is independent of
$(\mathcal X,W,\widetilde W)$. Then
\begin{equation}
\bigl(\widetilde A_n^{(w)},\widetilde d_n^{(w)}\bigr)
\D
\bigl(\widetilde A^{(w)},\widetilde d^{(w)}\bigr).
\label{eq:joint-A-d-limit-base}
\end{equation}
Moreover, $\det\widetilde A^{(w)}>0$ almost surely.
\end{proposition}

\begin{proof}
The joint process convergence \eqref{eq:joint-process-Rn}, including the required independence of $Z$, follows from Lemma~\ref{lem:joint-process-Rn}.

Fix $\zeta\in(0,1)$ and let $j_n(\zeta):=\lfloor n\zeta\rfloor+1$. Set
{\small
\[
\begin{aligned}
T_n
&:=\biggl(\frac1n\sum_{k=1}^n\frac{X_{k-1}}k,
\ \frac1{n^2}\sum_{k=1}^n\frac{X_{k-1}^2}k,
\ \frac1n\sum_{k=1}^n\frac{M_kX_{k-1}}k,
-\frac1{\sqrt n}\sum_{k=1}^n\frac{M_kV_{k-1}}k,
 \ Z_n\biggr)^{\mathsf T},\\[1mm]
T_{n,\zeta}
&:=\biggl(\frac1n\sum_{k=j_n(\zeta)}^n\frac{X_{k-1}}k,
\ \frac1{n^2}\sum_{k=j_n(\zeta)}^n\frac{X_{k-1}^2}k,
\ \frac1{(1+\beta)n}\sum_{k=j_n(\zeta)}^n\frac{M_kU_{k-1}}k,
-\frac1{\sqrt n}\sum_{k=j_n(\zeta)}^n\frac{M_kV_{k-1}}k,
 \ Z_n\biggr)^{\mathsf T}.
\end{aligned}
\]
}
The first two coordinates are weighted Riemann sums. For the score
coordinates, summation by parts gives
\[
\begin{aligned}
\frac1n\sum_{k=j_n(\zeta)}^n\frac{M_kU_{k-1}}k
&=\mathcal N_1^{(n)}
-\frac{n}{j_n(\zeta)}\mathcal N_{(j_n(\zeta)-1)/n}^{(n)}
+\sum_{k=j_n(\zeta)}^{n-1}\frac{n}{k(k+1)}\mathcal N_{k/n}^{(n)},\\
\frac1{\sqrt n}\sum_{k=j_n(\zeta)}^n\frac{M_kV_{k-1}}k
&=\mathcal P_1^{(n)}
-\frac{n}{j_n(\zeta)}\mathcal P_{(j_n(\zeta)-1)/n}^{(n)}
+\sum_{k=j_n(\zeta)}^{n-1}\frac{n}{k(k+1)}\mathcal P_{k/n}^{(n)}.
\end{aligned}
\]
Consequently, the continuous mapping theorem applied to \eqref{eq:joint-process-Rn} yields, for every fixed $\zeta>0$,
{\small
\[
\begin{aligned}
T_{n,\zeta}&\D T_\zeta,\\
T_\zeta
&:=\biggl(\int_\zeta^1\frac{\mathcal X_t}{t}\,dt,
\ \int_\zeta^1\frac{\mathcal X_t^2}{t}\,dt,
\ \sqrt{2\alpha\beta}\int_\zeta^1\frac{\mathcal X_t^{3/2}}t\,dW_t,
-2\beta\sqrt{\frac{\alpha}{1+\beta}}
\int_\zeta^1\frac{\mathcal X_t}t\,d\widetilde W_t,
 \ Z\biggr)^{\mathsf T},\\[1mm]
T
&:=\biggl(I_1,\ I_2,
\ \sqrt{2\alpha\beta}\int_0^1\frac{\mathcal X_t^{3/2}}t\,dW_t,
-2\beta\sqrt{\frac{\alpha}{1+\beta}}
\int_0^1\frac{\mathcal X_t}t\,d\widetilde W_t,
\ Z\biggr)^{\mathsf T}.
\end{aligned}
\]
}

Since $U_{k-1}=(1+\beta)X_{k-1}-\beta V_{k-1}$, the difference between the third coordinates of $T_n$ and $T_{n,\zeta}$ is
\[
\frac1{(1+\beta)n}\sum_{k=1}^{j_n(\zeta)-1}\frac{M_kU_{k-1}}k
+
\frac{\beta}{(1+\beta)n}\sum_{k=1}^n\frac{M_kV_{k-1}}k.
\]
The four initial-segment estimates in
\eqref{eq:initial-segment-design}--\eqref{eq:initial-segment-score}, together
with \eqref{eq:negligible-V-score}, therefore imply that, for every
$\varepsilon>0$,
\[
\lim_{\zeta\downarrow0}\limsup_{n\to\infty}
\PP\bigl(\lVert T_n-T_{n,\zeta}\rVert>\varepsilon\bigr)=0.
\]
As $\zeta\downarrow0$, monotone convergence for the first two coordinates and
It\^o isometry for the next two give $T_\zeta\to T$ in probability.
The fixed-$\zeta$ convergence $T_{n,\zeta}\D T_\zeta$, the preceding approximation, and $T_\zeta\to T$ satisfy the hypotheses of the converging-together theorem \citep[Theorem~3.2]{billingsley1999convergence}. Applying that theorem to the full stacked vector gives $T_n\D T$.

It remains only to assemble the entries of the design matrix. The $(1,1)$ entry is the second coordinate of $T_n$, while \eqref{eq:V2-weighted} gives
\[
\frac1n\sum_{k=1}^n\frac{V_{k-1}^2}{k}
=
\frac{2\beta}{1+\beta}\frac1n\sum_{k=1}^n\frac{X_{k-1}}k+\op(1).
\]
The $(3,3)$ entry equals one. The three off-diagonal entries converge to zero because
\[
\begin{aligned}
n^{-3/2}\sum_{k=1}^n\frac{X_{k-1}V_{k-1}}k
  &=\Op(n^{-1/2}),
 &\qquad
n^{-1}H_n^{-1/2}\sum_{k=1}^n\frac{X_{k-1}}k
  &=\Op(H_n^{-1/2}),\\
n^{-1/2}H_n^{-1/2}\sum_{k=1}^n\frac{V_{k-1}}k
  &=\Op\!\left(\sqrt{\frac{H_n}{n}}\right),
\end{aligned}
\]
which are all $\op(1)$, by \eqref{eq:XV-bounds} and $T_n\D T$. The last
three coordinates of $T_n$ are exactly $\widetilde d_n^{(w)}$. The continuous
mapping theorem together with Slutsky's lemma proves
\eqref{eq:joint-A-d-limit-base}.

Finally,
$\det\widetilde A^{(w)}=2\beta I_1I_2/(1+\beta)>0$ almost surely.
\end{proof}

\begin{proposition}[Transfer to fixed initial conditions and fixed offsets]
\label{prop:joint-convergence}
For an arbitrary fixed initial condition $(X_{-1},X_0)\in\mathbb Z_+^2$ and every
fixed $\ell\geq0$, let
\[
\mathsf{D}_{n,\ell}:=\diag\bigl(n^{-1},n^{-1/2},H_{n,\ell}^{-1/2}\bigr),
\quad
\widetilde A_{n,\ell}^{(w)}
:=\mathsf{D}_{n,\ell}A_{n,\ell}^{(w)}\mathsf{D}_{n,\ell},
\quad
\widetilde d_{n,\ell}^{(w)}:=\mathsf{D}_{n,\ell}d_{n,\ell}^{(w)}.
\]
Then
\[
\bigl(\widetilde A_{n,\ell}^{(w)},
\widetilde d_{n,\ell}^{(w)}\bigr)
\D
\bigl(\widetilde A^{(w)},\widetilde d^{(w)}\bigr),
\]
where the limit is the one in
Proposition~\ref{prop:joint-convergence-base}. Moreover,
$\PP\{\det A_{n,\ell}^{(w)}=0\}\to0$.
\end{proposition}

\needspace{4\baselineskip}
\begin{proof}
We first transfer the zero-start result at $\ell=0$. Couple the process
from a fixed initial pair with a zero-start process by sharing all immigrant
families and their reproduction marks. Their difference consists of the
descendants of the finitely many initial individuals. By
Lemma~\ref{lem:old-family-survival}, these descendants vanish after an almost
surely finite time $T$. The two count paths, regressors, and true residuals
therefore agree for all $k\geq T+2$.

Let superscripts $x$ and $0$ denote the fixed-start and zero-start systems.
For every fixed $\ell$, each entry of
$A_{n,\ell}^{(x)}-A_{n,\ell}^{(0)}$ and
$d_{n,\ell}^{(x)}-d_{n,\ell}^{(0)}$ is eventually a finite random constant,
hence $\Op(1)$. The corresponding normalized differences tend to zero
almost surely. In particular, this transfers
Proposition~\ref{prop:joint-convergence-base}, including the independence
of $Z$, to every fixed initial condition at $\ell=0$. Differences of the
auxiliary weighted sums are also finite sums, so their probability orders
are preserved.

It remains to transfer from offset zero to a fixed $\ell$. The weight difference satisfies
\[
\frac1{k+\ell}-\frac1k
=-\frac{\ell}{k(k+\ell)},
\qquad \frac{\ell}{k(k+\ell)}\leq \frac{\ell}{k^2}.
\]
For the following moment calculations, first take the zero-start process.
The moment bounds and martingale orthogonality give
\[
 \E\left(\sum_{k=1}^n\frac{-\ell M_kX_{k-1}}{k(k+\ell)}\right)^2
 =\sum_{k=1}^n\frac{\ell^2}{k^2(k+\ell)^2}\E(M_k^2X_{k-1}^2)
 =\mathrm O(H_n),
\]
whereas
\[
 \E\left(\sum_{k=1}^n\frac{-\ell M_kV_{k-1}}{k(k+\ell)}\right)^2
 =\mathrm O(1),
 \qquad
 \E\left(\sum_{k=1}^n\frac{-\ell M_k}{k(k+\ell)}\right)^2
 =\mathrm O(1).
\]
For the design, for instance,
$\E\lvert\sum_{k=1}^n\frac{-\ell X_{k-1}^2}{k(k+\ell)}\rvert
=\mathrm O(n)$; the other entries follow from Cauchy--Schwarz in the same
way. Consequently,
\begin{equation}
A_{n,\ell}^{(w)}-A_{n,0}^{(w)}
=
\begin{pmatrix}
\Op(n)&\Op(\sqrt n)&\Op(H_n)\\
\Op(\sqrt n)&\Op(H_n)&\Op(1)\\
\Op(H_n)&\Op(1)&\mathrm O(1)
\end{pmatrix},
\qquad
d_{n,\ell}^{(w)}-d_{n,0}^{(w)}
=
\begin{pmatrix}
\Op(\sqrt{H_n})\\
\Op(1)\\
\Op(1)
\end{pmatrix}.
\label{eq:offset-perturbation-orders}
\end{equation}
For a fixed nonzero initial condition, the difference between each of
these offset comparisons and its zero-start counterpart is again a finite
sum, hence $\Op(1)$. Thus all the probability orders in
\eqref{eq:offset-perturbation-orders} hold for every fixed initial pair.
Also,
\[
H_{n,\ell}=H_n-\sum_{k=1}^n\frac{\ell}{k(k+\ell)}
=H_n+\mathrm O(1),
\qquad \mathsf{D}_{n,\ell}\mathsf{D}_{n,0}^{-1}\longrightarrow I_3.
\]
It follows from~\eqref{eq:offset-perturbation-orders} that
\[
\mathsf{D}_{n,\ell}(A_{n,\ell}^{(w)}-A_{n,0}^{(w)})\mathsf{D}_{n,\ell}\Pto0,
\qquad
\mathsf{D}_{n,\ell}(d_{n,\ell}^{(w)}-d_{n,0}^{(w)})\Pto0.
\]
Slutsky's lemma now gives the asserted joint limit. Since
$\det\widetilde A^{(w)}>0$ almost surely, the portmanteau theorem gives the
asymptotic invertibility statement.
\end{proof}

\needspace{9\baselineskip}
\begin{proposition}[Joint OLS--constrained-WLS limit]
\label{prop:joint-ols-wls}
Under the assumptions of Theorem~\ref{thm:main}, for every fixed initial condition
and every fixed $\ell\geq0$,
\[
\begin{pmatrix}
 S_n^{\mathrm{OLS}}\\[1mm]
 \sqrt n(\widetilde\beta_{n,\ell}^c-\beta)\\[1mm]
 \sqrt{H_{n,\ell}}(\widetilde\mu_{n,\ell}^c-\mu)
\end{pmatrix}
\D
\begin{pmatrix}
 \Psi_\varrho\\[1mm]
 \Xi_w\\[1mm]
 Z
\end{pmatrix},
\qquad \Psi_\varrho\perp Z.
\]
\end{proposition}

\begin{proof}
Put
\[
\begin{aligned}
z_k&:=(X_{k-1},-V_{k-1},1)^{\mathsf T},
& A_n^{(o)}&:=\sum_{k=1}^n z_kz_k^{\mathsf T},
& d_n^{(o)}&:=\sum_{k=1}^nM_kz_k,\\
E_n&:=\diag(n^{-3/2},n^{-1},n^{-1/2}),
& F_n&:=\diag(n^{-2},n^{-3/2},n^{-1}).
\end{aligned}
\]
Let $(\widetilde A^{(o)},\widetilde d^{(o)})$ denote the normalized OLS
design and score limits in BIP's Theorem~4.1.
First take $X_{-1}=X_0=0$ and $\ell=0$. The identity
\[
 \sum_{k=1}^{\lfloor nt\rfloor}M_k
 =U_{\lfloor nt\rfloor}-\mu\lfloor nt\rfloor
 =(1+\beta)X_{\lfloor nt\rfloor}
   -\beta V_{\lfloor nt\rfloor}-\mu\lfloor nt\rfloor
\]
and $n^{-1}\max_{k\leq n}|V_k|\Pto0$ show that the process
$\mathcal M^{(n)}$ may be adjoined to
Lemma~\ref{lem:joint-process-Rn}, with limit
$\mathcal M_t=(1+\beta)\mathcal X_t-\mu t$. The normalized unweighted score
coordinates are terminal-value functionals of
$(\mathcal M^{(n)},\mathcal N^{(n)},\mathcal P^{(n)})$. The design entries
are Riemann sums of $\mathcal X^{(n)}$ together with the $\op(1)$ remainder
terms established in BIP's Theorem~4.1. Applying that calculation to the
joint convergence gives
\[
 (E_nA_n^{(o)}E_n,F_nd_n^{(o)},
   \widetilde A_n^{(w)},\widetilde d_n^{(w)})
 \D
 (\widetilde A^{(o)},\widetilde d^{(o)},
   \widetilde A^{(w)},\widetilde d^{(w)}).
\]
The fixed-start coupling in the proof of
Proposition~\ref{prop:joint-convergence} also applies to the unweighted sums:
each entry of $A_n^{(o,x)}-A_n^{(o,0)}$ and
$d_n^{(o,x)}-d_n^{(o,0)}$ is eventually a finite random constant. Hence
\[
 E_n(A_n^{(o,x)}-A_n^{(o,0)})E_n\longrightarrow0,
 \qquad
 F_n(d_n^{(o,x)}-d_n^{(o,0)})\longrightarrow0
 \quad\text{almost surely}.
\]
Together with Proposition~\ref{prop:joint-convergence}, this proves the
displayed joint design--score convergence for every fixed initial condition
and fixed $\ell$.

Finally, $F_n=n^{-1/2}E_n$. On the event that $A_n^{(o)}$ is nonsingular,
the OLS normal equations imply
\[
 \diag(n,\sqrt n,1)
 \bigl(\widehat\varrho_n-1,\widehat\beta_n-\beta,
 \widehat\mu_n-\mu\bigr)^{\mathsf T}
 =(E_nA_n^{(o)}E_n)^{-1}F_nd_n^{(o)}.
\]
The first coordinate of the limiting right-hand side is
$\Psi_\varrho$. Applying the unrestricted-WLS normal-equation map jointly to
the stacked convergence gives the joint limit with the normalized unrestricted
estimators of $(\beta,\mu)$, using the replacement
$\widetilde W\mapsto-\widetilde W$ from the proof of Theorem~\ref{thm:main}.
The rates in \eqref{eq:constrained-unrestricted-rates}, proved in
Appendix~\ref{app:constrained-comparison}, make their replacement by the
constrained estimators negligible after normalization. Slutsky's lemma
therefore gives the asserted joint limit. Since
$\det\widetilde A^{(o)}=
2\beta J_1(J_2-J_1^2)/(1+\beta)>0$ almost surely, all inverses used here exist
with probability tending to one, using also
Proposition~\ref{prop:joint-convergence} for the WLS inverse.
Independence follows because
$\Psi_\varrho$ is measurable with respect to $(\mathcal X,W)$, whereas $Z$
is independent of $(\mathcal X,W,\widetilde W)$.
\end{proof}

\section{Proof of Theorem~\ref{thm:main}}

\begin{proof}
Combining \eqref{eq:normal-canonical} with the scaling matrix $\mathsf{D}_{n,\ell}$
gives
\[
\begin{pmatrix}
n(\widetilde\varrho_{n,\ell}-1)\\[2mm]
\sqrt n\,(\widetilde\beta_{n,\ell}-\beta)\\[2mm]
\sqrt{H_{n,\ell}}\,(\widetilde\mu_{n,\ell}-\mu)
\end{pmatrix}
=
\bigl(\widetilde A_{n,\ell}^{(w)}\bigr)^{-1}
\widetilde d_{n,\ell}^{(w)}.
\]
By Proposition~\ref{prop:joint-convergence},
\[
\bigl(\widetilde A_{n,\ell}^{(w)},\widetilde d_{n,\ell}^{(w)}\bigr)
\D
\bigl(\widetilde A^{(w)},\widetilde d^{(w)}\bigr).
\]
Since $\det \widetilde A^{(w)}>0$ almost surely, the continuous mapping theorem yields
\[
\bigl(\widetilde A_{n,\ell}^{(w)}\bigr)^{-1}
\widetilde d_{n,\ell}^{(w)}
\D
\bigl(\widetilde A^{(w)}\bigr)^{-1}\widetilde d^{(w)}
=
\begin{pmatrix}
\dfrac{\sqrt{2\alpha\beta}\int_0^1 t^{-1}\mathcal X_t^{3/2}\,dW_t}{I_2}\\[4mm]
-\sqrt{\alpha(1+\beta)}\,\dfrac{\int_0^1 t^{-1}\mathcal X_t\,d\widetilde W_t}{I_1}\\[4mm]
Z
\end{pmatrix}.
\]
Because $-\widetilde W$ is again a standard Wiener process independent of
$(\mathcal X,W)$, we may rewrite the second coordinate with the sign convention
used in \eqref{eq:wls-limit-functionals}. This proves
\eqref{eq:main-canonical}.
\end{proof}

\section{Proofs of Theorems~\ref{thm:constrained-wls} and~\ref{thm:intercept-second-order}}
\label{app:constrained-comparison}
\begin{proof}
Fix $\ell\geq0$ and set
\[
 S_{X,n,\ell}:=\sum_{k=1}^n\frac{X_{k-1}}{k+\ell},
 \quad S_{V,n,\ell}:=\sum_{k=1}^n\frac{V_{k-1}}{k+\ell},
 \quad S_{XV,n,\ell}:=\sum_{k=1}^n\frac{X_{k-1}V_{k-1}}{k+\ell},
 \quad S_{VV,n,\ell}:=\sum_{k=1}^n\frac{V_{k-1}^2}{k+\ell}.
\]
The $(\beta,\mu)$ principal block and the corresponding $\varrho$ column
of $A_{n,\ell}^{(w)}$ are
\[
 C_{n,\ell}=\begin{pmatrix}
 S_{VV,n,\ell}&-S_{V,n,\ell}\\-S_{V,n,\ell}&H_{n,\ell}
 \end{pmatrix},
 \qquad
 b_{n,\ell}=\begin{pmatrix}-S_{XV,n,\ell}\\S_{X,n,\ell}\end{pmatrix}.
\]
By Proposition~\ref{prop:joint-convergence}, $A_{n,\ell}^{(w)}$ is
nonsingular with probability tending to one. On this event it is positive
definite, as is $C_{n,\ell}$, and the normal equations for
\eqref{eq:wls-criterion} and~\eqref{eq:wls-criterion-constrained} give
\[
 \begin{pmatrix}
 \widetilde\beta_{n,\ell}^c-\widetilde\beta_{n,\ell}\\
 \widetilde\mu_{n,\ell}^c-\widetilde\mu_{n,\ell}
 \end{pmatrix}
 =C_{n,\ell}^{-1}b_{n,\ell}(\widetilde\varrho_{n,\ell}-1).
\]
Lemma~\ref{lem:weighted-V-sums}, together with the fixed-start coupling and
fixed-offset bounds in the proof of Proposition~\ref{prop:joint-convergence},
gives $S_{V,n,\ell}=\Op(H_{n,\ell})$ and $S_{XV,n,\ell}=\Op(n)$.
The weighted-sum convergence established in the proofs of
Propositions~\ref{prop:joint-convergence-base} and~\ref{prop:joint-convergence},
combined with the normal-equation map used in Theorem~\ref{thm:main}, gives,
jointly with \eqref{eq:main-canonical},
\[
 \left(\frac{S_{X,n,\ell}}n,\frac{S_{VV,n,\ell}}n,
       \frac{\det C_{n,\ell}}{nH_{n,\ell}}\right)
 \D (I_1,c_\beta I_1,c_\beta I_1).
\]
Here $\det C_{n,\ell}=S_{VV,n,\ell}H_{n,\ell}-S_{V,n,\ell}^2$ and
$S_{V,n,\ell}^2/(nH_{n,\ell})=\Op(H_{n,\ell}/n)=\op(1)$.
Since $c_\beta I_1>0$ almost surely, the two coordinates of
$C_{n,\ell}^{-1}b_{n,\ell}$ are respectively $\Op(1)$ and
$\Op(n/H_{n,\ell})$. Theorem~\ref{thm:main} and the block identity imply
\begin{equation}
 \widetilde\beta_{n,\ell}^c-\widetilde\beta_{n,\ell}=\Op(n^{-1}),
 \qquad
 \widetilde\mu_{n,\ell}^c-\widetilde\mu_{n,\ell}=\Op(H_{n,\ell}^{-1}).
 \label{eq:constrained-unrestricted-rates}
\end{equation}
Multiplying these differences by $\sqrt n$ and $\sqrt{H_{n,\ell}}$,
respectively, and applying Slutsky's lemma proves
Theorem~\ref{thm:constrained-wls}.

For Theorem~\ref{thm:intercept-second-order}, the intercept coordinate of
the same block identity gives
\begin{equation}
 H_{n,\ell}(\widetilde\mu_{n,\ell}^c-\widetilde\mu_{n,\ell})
 =\frac{S_{VV,n,\ell}S_{X,n,\ell}-S_{V,n,\ell}S_{XV,n,\ell}}{n^2}
  \frac{nH_{n,\ell}}{\det C_{n,\ell}}
 \,n(\widetilde\varrho_{n,\ell}-1).
 \label{eq:exact-intercept-gap}
\end{equation}
Since $S_{V,n,\ell}S_{XV,n,\ell}/n^2=\Op(H_{n,\ell}/n)=\op(1)$,
the product of the two fractions converges jointly to $I_1$.
The continuous mapping theorem yields the asserted augmented joint limit,
with final coordinate $I_1\Psi_w$.
\end{proof}

\section{Proof of Proposition~\ref{prop:fixed-offset-intercept}}

\begin{proof}
It is enough to compare an arbitrary fixed offset $\ell$ with offset zero.
Write
\[
\widetilde{\boldsymbol\theta}_{n,\ell}
:=(\widetilde\varrho_{n,\ell},\widetilde\beta_{n,\ell},
\widetilde\mu_{n,\ell})^{\mathsf T},
\qquad
\boldsymbol\theta:=(1,\beta,\mu)^{\mathsf T}.
\]
Subtracting the offset and zero-offset unrestricted normal equations gives
\[
A_{n,\ell}^{(w)}
(\widetilde{\boldsymbol\theta}_{n,\ell}
-\widetilde{\boldsymbol\theta}_{n,0})
=
(d_{n,\ell}^{(w)}-d_{n,0}^{(w)})
-(A_{n,\ell}^{(w)}-A_{n,0}^{(w)})
(\widetilde{\boldsymbol\theta}_{n,0}-\boldsymbol\theta).
\]
Theorem~\ref{thm:main} and~\eqref{eq:offset-perturbation-orders} make the
right-hand side
$(\Op(\sqrt{H_n}),\Op(1),\Op(1))^{\mathsf T}$. The normalized design limit implies that
the three entries in the third row of $(A_{n,\ell}^{(w)})^{-1}$ have orders
$\Op\!\left(1/(n\sqrt{H_n})\right)$,
$\Op\!\left(1/\sqrt{nH_n}\right)$, and $\Op(H_n^{-1})$,
respectively. Hence
$\widetilde\mu_{n,\ell}-\widetilde\mu_{n,0}=\Op(H_n^{-1})$.
The constrained two-dimensional normal equations give
$\widetilde\mu_{n,\ell}^c-
\widetilde\mu_{n,0}^c=\Op(H_n^{-1})$ by the same argument. Applying both
comparisons to $\ell_1$ and $\ell_2$, then using the triangle inequality and
$H_n\sim\log n$, proves the proposition.
\end{proof}

\Needspace{12\baselineskip}
\section{Proof of Proposition~\ref{prop:fixed-stationary}}
\label{app:fixed-stationary}

\begin{lemma}[Fixed-offset logarithmic ergodic law]
\label{lem:fixed-offset-log-ergodic}
Suppose that a sequence $(Y_k)$ satisfies
\[
 \frac1n\sum_{k=1}^nY_k\longrightarrow m
 \qquad\text{almost surely.}
\]
Then, for every fixed $\ell\geq0$,
\[
 \frac1{H_{n,\ell}}
 \sum_{k=1}^n \frac{Y_k}{k+\ell}
 \longrightarrow m
 \qquad\text{almost surely.}
\]
\end{lemma}

\begin{proof}
Set $B_k:=k^{-1}\sum_{j=1}^kY_j$. Abel summation gives
\[
 \frac1{H_{n,\ell}}\sum_{k=1}^n\frac{Y_k}{k+\ell}
 =\frac{nB_n}{(n+\ell)H_{n,\ell}}
  +\sum_{k=1}^{n-1}
    \frac{kB_k}{(k+\ell)(k+\ell+1)H_{n,\ell}}.
\]
The coefficients on the right are nonnegative, sum to one, and tend to zero
for each fixed index. Toeplitz's lemma applies on every sample path on
which $B_k\to m$.
\end{proof}

\begin{proof}[Proof of Proposition~\ref{prop:fixed-stationary}]
For $\varrho=\alpha+\beta<1$, the pair
$\mathbf X_k=(X_k,X_{k-1})^{\mathsf T}$ is a subcritical two-type
Galton--Watson process with immigration. Its type-specific offspring vectors
are $(B_\alpha,1)^{\mathsf T}$ and $(B_\beta,0)^{\mathsf T}$, where
$B_\alpha$ and $B_\beta$ are Bernoulli variables with means $\alpha$ and
$\beta$, and its immigration vector is $(\varepsilon_k,0)^{\mathsf T}$.
The mean operator on column population vectors is
$\left(\begin{smallmatrix}\alpha&\beta\\1&0\end{smallmatrix}\right)$,
whose spectral radius is less than one. The mean-matrix convention in
\citet{szucs2024ergodic} is the transpose of this column-vector convention.
The bounded offspring vectors of both types satisfy the final
arbitrary-state proviso of \citet[Theorem~2.5]{szucs2024ergodic}
with moment order eight.

By \citet[Theorems~2.1--2.2 and~2.5]{szucs2024ergodic} and the discussion
following Theorem~2.2 there, the chain has a
unique invariant law $\pi$, is positive Harris recurrent, and satisfies
\begin{equation}
 \sup_{k\geq-1}\E(X_k^8)<\infty
 \label{eq:fixed-stationary-uniform-eighth}
\end{equation}
from every fixed initial pair. Theorem~2.5 also gives finite eighth moments
under $\pi$ and convergence to the stationary expectations for functions
growing at most as the eighth power of the state norm. Positive Harris
recurrence and \citet[Theorem~17.0.1(i)]{meyn1993} give
\[
 \frac1n\sum_{k=1}^n f(\mathbf X_k)
 \longrightarrow\int f\,d\pi
 \qquad\text{almost surely}
\]
from every fixed initial pair whenever $\int|f|\,d\pi<\infty$.
Together with the stationary moment bound, this supplies the ergodic
strong laws used below.

Write $\E_\pi$ for stationary expectation and let
$\gamma_j=\Cov_\pi(X_k,X_{k-j})$. The conditional mean equation gives
$\gamma_1=\alpha\gamma_0+\beta\gamma_1$, so
\[
 \Cov_\pi\begin{pmatrix}X_{k-1}\\X_{k-2}\end{pmatrix}
 =\gamma_0\begin{pmatrix}1&a\\a&1\end{pmatrix},
 \qquad a=\frac{\alpha}{1-\beta}\in(0,1).
\]
The law of total variance and $\E_\pi X_k=\mu/(1-\varrho)$ imply
\[
 \gamma_0\geq\E_\pi\{\Var(X_k\mid\mathcal F_{k-1})\}
 =\{\alpha(1-\alpha)+\beta(1-\beta)\}\frac{\mu}{1-\varrho}
  +\Var(\varepsilon_1)>0.
\]
The lag covariance matrix is therefore positive definite. Set
\[
 z_k:=(X_{k-1},-V_{k-1},1)^{\mathsf T},\qquad
 Q_z:=\E_\pi(z_kz_k^{\mathsf T}).
\]
The moment bound makes $Q_z$ finite. The vector $z_k$ is an invertible
linear transformation of the lag pair augmented by the constant one.
The positive-definite lag covariance therefore implies $Q_z>0$. Hence
Lemma~\ref{lem:fixed-offset-log-ergodic} applies entrywise to the design
moments in both WLS systems.

Let
$\vartheta_0:=(\varrho,\beta,\mu)^{\mathsf T}$ and
$\widetilde\vartheta_{n,\ell}:=(\widetilde\varrho_{n,\ell},
\widetilde\beta_{n,\ell},\widetilde\mu_{n,\ell})^{\mathsf T}$.
The general canonical
equation~\eqref{eq:canonical} is
\[
 X_k=z_k^{\mathsf T}\vartheta_0+M_k,
 \qquad
 \E(M_k\mid\mathcal F_{k-1})=0.
\]
Define
\[
 A_{n,\ell}^{(w)}:=\sum_{k=1}^n
 \frac{z_kz_k^{\mathsf T}}{k+\ell},
 \qquad
 d_{n,\ell}^{(w)}:=\sum_{k=1}^n \frac{z_kM_k}{k+\ell}.
\]
The logarithmic ergodic law gives
\[
 \frac{A_{n,\ell}^{(w)}}{H_{n,\ell}}
 \longrightarrow Q_z
 \qquad\text{almost surely.}
\]
Moreover, $(d_{n,\ell}^{(w)})_{n\geq1}$ is a square-integrable vector
martingale. Indeed, writing
\[
 \sigma_k^2:=\E(M_k^2\mid\mathcal F_{k-1})
 =\alpha(1-\alpha)X_{k-1}
  +\beta(1-\beta)X_{k-2}
  +\Var(\varepsilon_1),
\]
we have
\[
 \E\!\left(M_k^2\lVert z_k\rVert^2\right)
 =\E\!\left(\sigma_k^2\lVert z_k\rVert^2\right)
 \leq C\E(1+X_{k-1}^3+X_{k-2}^3).
\]
Consequently,~\eqref{eq:fixed-stationary-uniform-eighth} and martingale
orthogonality give
\[
 \sup_{n\geq1}\E\bigl\lVert d_{n,\ell}^{(w)}\bigr\rVert^2
 =\sup_{n\geq1}\sum_{k=1}^n \frac1{(k+\ell)^2}
   \E\!\left(M_k^2\lVert z_k\rVert^2\right)
 <\infty,
\]
because $\sum_{k\geq1}(k+\ell)^{-2}<\infty$. The martingale convergence
theorem therefore yields
\[
 d_{n,\ell}^{(w)}
 \longrightarrow d_{\infty,\ell}^{(w,x)}
 \qquad\text{almost surely and in }L^2.
\]

The limiting score is nondegenerate. To see this, martingale orthogonality
gives
\[
 \Cov(d_{\infty,\ell}^{(w,x)})
 =\sum_{k=1}^{\infty}\frac1{(k+\ell)^2}
   \E(\sigma_k^2 z_kz_k^{\mathsf T}).
\]
Under the invariant law, $\sigma_k^2>0$ almost surely: this is immediate if
$\Var(\varepsilon_1)>0$, while if $\Var(\varepsilon_1)=0$, then
$\varepsilon_1=\mu>0$ almost surely and the stationary counts are positive.
Since $Q_z>0$ and $\sigma_k^2>0$ almost surely under $\pi$, the matrix
$\E_\pi(\sigma_k^2 z_kz_k^{\mathsf T})$ is positive definite. Each of its
integrands is a polynomial of degree at most three in the lag pair, so
\citet[Theorem~2.5]{szucs2024ergodic} gives
$\E(\sigma_k^2 z_kz_k^{\mathsf T})\to\E_\pi(\sigma_k^2 z_kz_k^{\mathsf T})$.
Hence $\Cov(d_{\infty,\ell}^{(w,x)})$ is positive definite.

Finally, positive definiteness of $Q_z$ implies eventual almost-sure
invertibility, and the unrestricted normal equations give
\[
 H_{n,\ell}(\widetilde\vartheta_{n,\ell}-\vartheta_0)
 =\left(\frac{A_{n,\ell}^{(w)}}{H_{n,\ell}}\right)^{-1}
   d_{n,\ell}^{(w)}
 \xrightarrow{\mathrm{a.s.}}Q_z^{-1}d_{\infty,\ell}^{(w,x)}
 =:\Lambda_{\ell,x}.
\]
Because $Q_z$ and $\Cov(d_{\infty,\ell}^{(w,x)})$ are positive definite, so is
$\Cov(\Lambda_{\ell,x})$. This proves the claimed unrestricted rate
statement; since $H_{n,\ell}\to\infty$, it also gives the first limit
in~\eqref{eq:fixed-stationary-limits}.

For the constrained estimator, put
\[
 r_k:=(-V_{k-1},1)^{\mathsf T},
 \qquad
 L_c(b,m):=\E_\pi(V_k+bV_{k-1}-m)^2.
\]
Stationarity gives $\E_\pi(V_k)=\E_\pi(V_{k-1})=0$. Moreover,
positive definiteness of $Q_z$ implies
$\sigma_V^2:=\E_\pi(V_{k-1}^2)>0$. Therefore the unique population projection
satisfies
\begin{equation}
 m^\dagger=0,
 \qquad
 b^\dagger
 =-\frac{\E_\pi(V_kV_{k-1})}{\sigma_V^2},
 \label{eq:fixed-stationary-projection}
\end{equation}

Subtracting $X_{k-1}$ from~\eqref{eq:canonical} gives
\[
 V_k=(\varrho-1)X_{k-1}-\beta V_{k-1}+\mu+M_k.
\]
The martingale-difference property gives
$\E_\pi(M_kV_{k-1})=0$, while
\[
 2X_{k-1}V_{k-1}
 =V_{k-1}^2+X_{k-1}^2-X_{k-2}^2
\]
and stationarity imply
$\E_\pi(X_{k-1}V_{k-1})=\sigma_V^2/2$. Consequently,
\[
 \E_\pi(V_kV_{k-1})
 =(\varrho-1)\E_\pi(X_{k-1}V_{k-1})
   -\beta\E_\pi(V_{k-1}^2)
 =-\left\{\beta+\frac{1-\varrho}{2}\right\}\sigma_V^2.
\]
Combining this with~\eqref{eq:fixed-stationary-projection} gives
$b^\dagger=\beta+(1-\varrho)/2$, as asserted in
\eqref{eq:fixed-stationary-limits}.

Finally,
\[
 Q_r:=\E_\pi(r_kr_k^{\mathsf T})
 =\begin{pmatrix}\sigma_V^2&0\\0&1\end{pmatrix}
\]
is positive definite. Lemma~\ref{lem:fixed-offset-log-ergodic} applied to
the constrained design moments gives
\[
 \frac1{H_{n,\ell}}
 \sum_{k=1}^n \frac{r_kr_k^{\mathsf T}}{k+\ell}\longrightarrow Q_r
 \qquad\text{almost surely.}
\]
For the response moments, the decomposition of $V_k$ above gives
\[
 \frac1{H_{n,\ell}}\sum_{k=1}^n\frac{r_kV_k}{k+\ell}
 =\frac1{H_{n,\ell}}\sum_{k=1}^n
   \frac{r_k\{(\varrho-1)X_{k-1}-\beta V_{k-1}+\mu\}}{k+\ell}
 +\frac1{H_{n,\ell}}\sum_{k=1}^n\frac{r_kM_k}{k+\ell}.
\]
The entries of the predictable term are polynomials of degree at most two
in the lag pair $\mathbf X_{k-1}$, so the logarithmic ergodic law applies to
the first sum. The unnormalized martingale sum in the second term consists of the
last two components of $d_{n,\ell}^{(w)}$, which converges almost surely
as proved above. Since $H_{n,\ell}\to\infty$, the second term tends to zero
almost surely. Using $\E_\pi(r_kM_k)=0$ and the stationary identities above,
we obtain
\[
 \frac1{H_{n,\ell}}
 \sum_{k=1}^n \frac{r_kV_k}{k+\ell}
 \longrightarrow\E_\pi(r_kV_k)=Q_r
 \begin{pmatrix}b^\dagger\\0\end{pmatrix}
\]
almost surely. This proves eventual invertibility and the constrained limit in
\eqref{eq:fixed-stationary-limits}.

The scalar intercept order used to justify the selection condition follows
directly from summation by parts. For $j\in\{0,1\}$,
\[
\sum_{k=1}^n w_{k,\ell}V_{k-j}
=w_{n,\ell}X_{n-j}-w_{1,\ell}X_{-j}
+\sum_{k=1}^{n-1}(w_{k,\ell}-w_{k+1,\ell})X_{k-j}
=\Op(1),
\]
where the final order follows from
\eqref{eq:fixed-stationary-uniform-eighth}. The intercept normal equation is
\[
H_{n,\ell}\widetilde\mu_{n,\ell}^c
=\sum_{k=1}^n w_{k,\ell}
 \bigl(V_k+\widetilde\beta_{n,\ell}^cV_{k-1}\bigr).
\]
Since $\widetilde\beta_{n,\ell}^c\Pto\beta+(1-\varrho)/2$, the right-hand side
is $\Op(1)$, proving
$\widetilde\mu_{n,\ell}^c=\Op(H_{n,\ell}^{-1})$.
\end{proof}

\section{Proof of Theorem~\ref{thm:ols-test}}

\begin{proof}
Under $H_0$, Theorems~\ref{thm:main} and
\ref{thm:constrained-wls} give
\[
(\widetilde\beta_{n,\ell},\widetilde\mu_{n,\ell})
\Pto(\beta,\mu),
\qquad
(\widetilde\beta_{n,\ell}^c,\widetilde\mu_{n,\ell}^c)
\Pto(\beta,\mu),
\]
so both WLS fits are admissible with probability tending to one. On the
constrained-admissibility event, continuity of $\delta(\cdot,\cdot)$ gives
$\widetilde\delta_{n,\ell}^C\Pto\delta(\beta,\mu)>0$.
Continuity of the quantile map gives, on the respective admissibility events,
\[
q_{\tau,n,\ell}^{\mathrm{OLS},C}
\Pto q^{\mathrm{OLS}},
\qquad
q_{\tau,n,\ell}^{\mathrm{OLS},W}
\Pto q^{\mathrm{OLS}},
\qquad
q^{\mathrm{OLS}}:=q_\tau^{\mathrm{OLS}}(\beta,\mu).
\]
Thus the operational selector converges to $q^{\mathrm{OLS}}$ regardless of
which of its two asymptotically admissible WLS branches is selected.
Together with $S_n^{\mathrm{OLS}}\D\Psi_\varrho$, Slutsky's lemma
and the no-atom condition yield
\[
\PP\{S_n^{\mathrm{OLS}}<q_{\tau,n,\ell}^{\mathrm{OLS}}\}
\longrightarrow
\PP(\Psi_\varrho<q^{\mathrm{OLS}})=\tau.
\]
Now fix $\varrho<1$. By Proposition~\ref{prop:fixed-stationary},
$\widetilde\mu_{n,\ell}^c=\Op(H_{n,\ell}^{-1})$ and
$\widetilde\beta_{n,\ell}^c\Pto\beta+(1-\varrho)/2\in(0,1)$. Consequently,
\[
\widetilde\delta_{n,\ell}^C
=\Op(H_{n,\ell}^{-1})
=\op(\eta_{n,\ell})
\]
whenever the constrained fit is admissible, so the probability that the first
branch of \eqref{eq:ols-critical-value} is selected tends to zero.
Proposition~\ref{prop:fixed-stationary} gives
\[
(\widetilde\beta_{n,\ell},\widetilde\mu_{n,\ell})
\Pto(\beta,\mu)\in\Theta_0.
\]
Moreover, standard stationary OLS theory gives, on the nonsingular OLS event,
\[
\frac{S_n^{\mathrm{OLS}}}{n}
=\widehat\varrho_n-1\Pto\varrho-1<0.
\]
Thus the unrestricted-WLS fit is admissible with probability tending to one,
and local boundedness of the quantile map makes the selected critical value
$\Op(1)$. Since $S_n^{\mathrm{OLS}}/n\Pto\varrho-1<0$, the statistic
diverges to $-\infty$ at rate $n$, proving consistency of the proposed test.
The OLS design matrix is nonsingular with probability tending to one under
both regimes, so assigning nonrejection on its singular event does not alter
either limit.
\end{proof}

\section{Proof of Lemma~\ref{lem:residual-score-variance}}
\label{app:residual-score-variance}

\begin{proof}
Fix $\ell\geq0$ and put $w_k=(k+\ell)^{-1}$ and $H=H_{n,\ell}$.
Under the exact unit root,
\[
 \E(M_k^2\mid\mathcal F_{k-1})
 =\alpha\beta(X_{k-1}+X_{k-2})+\Var(\varepsilon_1).
\]
At zero start,
Lemma~\ref{lem:weighted-V-k2}, $H/H_n\to1$, and
$w_k^2-k^{-2}=O(k^{-3})$ give
\[
 \frac1H\sum_{k=1}^n w_k^2X_{k-r}\Pto\frac\mu{1+\beta},
 \qquad r=1,2.
\]
The centered squares
$M_k^2-\E(M_k^2\mid\mathcal F_{k-1})$ are orthogonal and
$\E(M_k^4)=O(k^2)$, so their weighted sum divided by $H$ has second moment
bounded by $CH^{-2}\sum k^{-2}=o(1)$. The conditional variance identity
therefore yields
\begin{equation}
 Q_n:=\frac1H\sum_{k=1}^n w_k^2M_k^2\Pto\sigma_\mu^2.
 \label{eq:residual-score-true-variance}
\end{equation}
Use the fixed-start coupling from the proof of
Proposition~\ref{prop:joint-convergence}, with superscripts $x$ and $0$
denoting the fixed-start and zero-start processes. Let $T$ be the almost
surely finite extinction time supplied there by
Lemma~\ref{lem:old-family-survival}, and put $\kappa=T+2$. For all $k\geq\kappa$,
$M_k^x=M_k^0$ and $V_{k-1}^x=V_{k-1}^0$. Hence, for $n\geq\kappa$,
\[
\begin{aligned}
 Q_n^x-Q_n^0
 &=\frac1H\sum_{k=1}^{\kappa-1}w_k^2
   \{(M_k^x)^2-(M_k^0)^2\}
 =\frac{C_\ell(\omega)}H\longrightarrow0
 \quad\text{almost surely},\\
 C_\ell(\omega)
 &:=\sum_{k=1}^{\kappa-1}w_k^2\{(M_k^x)^2-(M_k^0)^2\}.
\end{aligned}
\]
The random variable $C_\ell$ is finite almost surely, so
\eqref{eq:residual-score-true-variance} holds from every fixed initial pair.
At zero start, \eqref{eq:moment-bounds} gives
\[
 \E\!\left\{\sum_{k=1}^n w_k^2(V_{k-1}^0)^2\right\}
 \leq C\sum_{k=1}^n\frac{k}{(k+\ell)^2}\leq CH,
\]
and hence this sum is $\Op(H)$. Similarly, for $n\geq\kappa$,
\[
\begin{aligned}
 \frac1H\sum_{k=1}^n w_k^2\{(V_{k-1}^x)^2-(V_{k-1}^0)^2\}
 &=\frac{D_\ell(\omega)}H\longrightarrow0
 \quad\text{almost surely},\\
 D_\ell(\omega)
 &:=\sum_{k=1}^{\kappa-1}w_k^2
   \{(V_{k-1}^x)^2-(V_{k-1}^0)^2\}.
\end{aligned}
\]
Since $D_\ell$ is finite almost surely, the bound
$\sum_{k=1}^n w_k^2V_{k-1}^2=\Op(H)$ also holds from every fixed initial pair.

Let $e_k=\widehat M_{k,\ell}-M_k$. The rates from
Theorem~\ref{thm:constrained-wls} and
$e_k=(\widetilde\beta_{n,\ell}^c-\beta)V_{k-1}
-(\widetilde\mu_{n,\ell}^c-\mu)$ imply
\[
\begin{aligned}
 R_n:=H^{-1}\sum_{k=1}^n w_k^2e_k^2
 &\leq\frac2H(\widetilde\beta_{n,\ell}^c-\beta)^2
             \sum_{k=1}^n w_k^2V_{k-1}^2
 +\frac2H(\widetilde\mu_{n,\ell}^c-\mu)^2\sum_{k=1}^n w_k^2\\
 &=\Op(n^{-1}+H^{-2})=\op(1).
\end{aligned}
\]
Cauchy--Schwarz bounds the absolute difference between the residual
variance estimate and $Q_n$ by $2\sqrt{Q_nR_n}+R_n=o_{\mathbb P}(1)$.
\end{proof}
\section{Proof of Corollary~\ref{cor:pretest-joint}}
\begin{proof}
Put
$Y_{n,\ell}^c:=\sqrt{H_{n,\ell}}
(\widetilde\mu_{n,\ell}^c-\mu)$. Proposition~\ref{prop:joint-ols-wls},
consistency of the feasible OLS critical value in
Theorem~\ref{thm:ols-test}, and Slutsky's lemma give
\[
(S_n^{\mathrm{OLS}},Y_{n,\ell}^c,
q_{\tau,n,\ell}^{\mathrm{OLS}})
\D(\Psi_\varrho,Z,q^{\mathrm{OLS}}).
\]
Since $q^{\mathrm{OLS}}$ is the lower $\tau$-quantile and
$\PP(\Psi_\varrho=q^{\mathrm{OLS}})=0$,
\[
\PP(A_{n,\ell}^{\mathrm{OLS}})\longrightarrow
p_\tau:=\PP(\Psi_\varrho\ge q^{\mathrm{OLS}})=1-\tau>0.
\]
For every $x\in\mathbb R$, the boundary of the corresponding bivariate event
is contained in
$\{Z=x\}\cup\{\Psi_\varrho=q^{\mathrm{OLS}}\}$ and therefore has
probability zero. The portmanteau theorem and
$Z\perp\Psi_\varrho$ yield
\[
\PP(Y_{n,\ell}^c\le x\mid A_{n,\ell}^{\mathrm{OLS}})
=\frac{\PP(Y_{n,\ell}^c\le x,A_{n,\ell}^{\mathrm{OLS}})}
{\PP(A_{n,\ell}^{\mathrm{OLS}})}
\longrightarrow
\frac{\PP(Z\le x,\Psi_\varrho\ge q^{\mathrm{OLS}})}{p_\tau}
=\PP(Z\le x),
\]
which proves the conditional weak convergence.

Lemma~\ref{lem:residual-score-variance} gives
\[
\widehat\sigma_{\mu,n,\ell}^{2,\mathrm{RS}}\Pto\sigma_\mu^2.
\]
Because $\PP(A_{n,\ell}^{\mathrm{OLS}})\to1-\tau>0$, for every $\epsilon>0$,
\[
\PP\!\left(
|\widehat\sigma_{\mu,n,\ell}^{2,\mathrm{RS}}-\sigma_\mu^2|>\epsilon
\,\middle|\,A_{n,\ell}^{\mathrm{OLS}}
\right)
\le
\frac{\PP(|\widehat\sigma_{\mu,n,\ell}^{2,\mathrm{RS}}-
\sigma_\mu^2|>\epsilon)}
{\PP(A_{n,\ell}^{\mathrm{OLS}})}
\longrightarrow0.
\]
Thus the studentizer is also consistent conditional on nonrejection. The
conditional weak convergence, Slutsky's lemma, and continuity of the standard
normal distribution give the stated coverage limit for $\mu$. The delta-method
expansion in Section~\ref{sec:gaussian-inference} remains valid conditionally
on nonrejection, while
$\widehat\sigma_{g,n,\ell}^{2,\mathrm{RS}}\Pto\sigma_g^2$; this gives the
stated coverage limit for $g$.
\end{proof}

\section{Sensitivity to the calibration threshold}
\label{app:ols-test-tuning}

\begin{table}[H]
\centering
\scriptsize
\setlength{\tabcolsep}{2pt}
\renewcommand{\arraystretch}{0.95}
\caption{Sensitivity of the constrained-WLS/unrestricted-WLS selector to the threshold multiplier.}
\label{tab:mc-wls-fallback-sensitivity}
\begin{tabular}{lrrcccc}
\toprule
Regime & $n$ & $\varrho$ & $c_{\mathrm{sel}}=.1$ & $c_{\mathrm{sel}}=.5$ & $c_{\mathrm{sel}}=1$ & $c_{\mathrm{sel}}=2$ \\
\midrule
Transient & 100 & 1 & 6.82 (0.11) / 99.8 & 6.82 (0.11) / 99.8 & 6.82 (0.11) / 99.8 & 6.82 (0.11) / 99.8 \\
& & .95 & 49.50 (0.22) / 99.6 & 49.50 (0.22) / 99.6 & 49.50 (0.22) / 99.6 & 49.50 (0.22) / 99.6 \\
& 1000 & 1 & 5.02 (0.10) / 100.0 & 5.02 (0.10) / 100.0 & 5.02 (0.10) / 100.0 & 5.02 (0.10) / 100.0 \\
& & .95 & 100.00 (0.00) / 100.0 & 100.00 (0.00) / 100.0 & 100.00 (0.00) / 100.0 & 100.00 (0.00) / 100.0 \\
\addlinespace
Boundary & 100 & 1 & 4.79 (0.10) / 99.0 & 5.00 (0.10) / 96.9 & 5.52 (0.10) / 91.2 & 6.75 (0.11) / 75.9 \\
& & .95 & 9.47 (0.13) / 98.9 & 10.19 (0.14) / 93.5 & 12.03 (0.15) / 79.7 & 16.31 (0.17) / 54.0 \\
& 1000 & 1 & 3.87 (0.09) / 100.0 & 3.87 (0.09) / 100.0 & 3.93 (0.09) / 99.0 & 4.48 (0.09) / 90.1 \\
& & .95 & 57.76 (0.22) / 100.0 & 63.21 (0.22) / 93.1 & 80.89 (0.18) / 72.8 & 94.26 (0.10) / 44.1 \\
\addlinespace
Recurrent & 100 & 1 & 5.31 (0.10) / 93.8 & 5.94 (0.11) / 86.7 & 6.85 (0.11) / 75.6 & 8.35 (0.12) / 56.9 \\
& & .95 & 8.86 (0.13) / 93.7 & 10.97 (0.14) / 77.4 & 13.51 (0.15) / 58.4 & 17.55 (0.17) / 36.5 \\
& 1000 & 1 & 4.02 (0.09) / 100.0 & 4.08 (0.09) / 99.0 & 4.40 (0.09) / 93.4 & 5.41 (0.10) / 73.9 \\
& & .95 & 42.12 (0.22) / 99.4 & 60.87 (0.22) / 74.9 & 81.53 (0.17) / 50.7 & 90.35 (0.13) / 28.6 \\
\bottomrule
\end{tabular}

\begin{minipage}{0.98\linewidth}
\footnotesize
\emph{Notes:} The selector uses constrained WLS when its nuisance pair is
admissible and $\widetilde\delta_{n,\ell}^C>c_{\mathrm{sel}}H_{n,\ell}^{-1/2}$; otherwise it
uses unrestricted WLS when admissible. If the selected calibration is
unavailable, the test does not reject. Each entry reports rejection percentage (Monte
Carlo standard error in percentage points) followed by the constrained-branch
percentage. All cells use $50{,}000$ replications at nominal level $5\%$.
\end{minipage}
\end{table}

\end{document}